\documentclass[abstract=on]{scrartcl}
\usepackage{amsmath,amsthm,amssymb}
\usepackage[utf8]{inputenc}
\usepackage{uniinput}
\usepackage[ngerman,english]{babel}
\usepackage{ellipsis,ragged2e,microtype}
\usepackage{graphicx}
\usepackage{mathtools}
\usepackage{tikz}
\usepackage{pgfplots}
\usepackage{caption}
\usepackage{subcaption}
\usepackage{paralist}
\usepackage{color}
\usepackage{url}
\usepackage{ifthen}
\usepackage{esint}
\usepackage[nonumberlist,toc]{glossaries}
\usepackage{psfrag}
\usetikzlibrary{patterns}
\makeglossaries
\newcommand{\ssym}{\text{\textnormal{sym}}}

\newcommand{\eps}{\epsilon}
\newcommand{\sca}{\mathcal{A}}
\newcommand{\sch}{\mathcal{H}}
\newcommand{\1}{\mathbf{1}}
\newcommand{\R}{\mathbb{R}}
\newcommand{\om}{\Omega}
\newcommand{\us}{u_{\sigma}} 
\renewcommand{\d}{\,\text{d}}
\newcommand{\st}{\;:\;}
\newcommand{\ton}{\text{on }}
\newcommand{\tin}{\text{in }}

\newcommand{\cof}{\operatorname{cof}}

\renewcommand{\div}{\mathop{\text{div}}}

\renewcommand{\inf}{\mathop{\textnormal{inf}}}
\newcommand{\essinf}{\mathop{\textnormal{ess inf}}}
\newcommand{\meas}{\mathop{\text{meas}}}
\newcommand{\dist}{\mathop{\text{dist}}}

\newcommand{\tr}{\mathop{\text{tr}}}
\newcommand{\id}{\mathop{\text{id}}}

\newcommand{\wtilde}{\widetilde}
\newcommand{\etR}{\wtilde{e_r}}
\newcommand{\etT}{\wtilde{e_θ}}
\newcommand*\xbar[1]{%
  \hbox{%
    \vbox{%
      \hrule height 0.5pt 
      \kern0.4ex
      \hbox{%
        \kern-0.04em
        \ensuremath{#1}%
        \kern-0.04em
      }%
    }%
  }%
}
\renewcommand{\bar}{\xbar}
\renewcommand{\hat}{\widehat}

\newcommand{\Div}{\textrm{div}\,}

\newcommand*{\Cdot}[1][1.5]{%
  \mathpalette{\CdotAux{#1}}\cdot%
}
\newdimen\CdotAxis
\newcommand*{\CdotAux}[3]{%
  {%
    \settoheight\CdotAxis{$#2\vcenter{}$}%
    \sbox0{%
      \raisebox\CdotAxis{%
        \scalebox{#1}{%
          \raisebox{-\CdotAxis}{%
            $\mathsurround=0pt #2#3$%
          }%
        }%
      }%
    }%
    \dp0=0pt %
    \sbox2{$#2\bullet$}%
    \ifdim\ht2<\ht0 %
      \ht0=\ht2 %
    \fi
    \sbox2{$\mathsurround=0pt #2#3$}%
    \hbox to \wd2{\hss\usebox{0}\hss}%
  }%
} 

\newtheorem{theorem}{Theorem}

\newtheorem{proposition}[theorem]{Proposition}
\newtheorem{remark}[theorem]{Remark}

\newtheorem{lemma}[theorem]{Lemma}
\newtheorem{corollary}[theorem]{Corollary}
\title{Twists and shear maps in nonlinear elasticity: explicit solutions and vanishing Jacobians}
\author{Jonathan J. Bevan\footnote{Department of Mathematics, University of Surrey, Guildford, UK. e: j.bevan@surrey.ac.uk} \ and Sandra Käbisch$^*$}

\begin{document}
\maketitle
\begin{abstract}
\noindent In this paper we study constrained variational problems that are principally motivated by nonlinear elasticity theory.   We examine in particular the relationship between the positivity of the Jacobian $\det \nabla u$ and the uniqueness and regularity of energy minimizers $u$ that are either twist maps or shear maps.  We exhibit \emph{explicit} twist maps, defined on two-dimensional annuli, that are stationary points of an appropriate energy functional and whose Jacobian vanishes on a set of positive measure in the annulus.  Within the class of shear maps we precisely characterize the unique global energy minimizer $u_{\sigma}: \Omega\to \mathbb{R}^2$ in a model, two-dimensional case.  The shear map minimizer has the properties that (i) $\det \nabla u_{\sigma}$ is strictly positive on one part of the domain $\Omega$, (ii) $\det \nabla u_{\sigma} = 0$ necessarily holds on the rest of $\Omega$, and (iii) properties (i) and (ii) combine to ensure that $\nabla u_{\sigma}$ is not continuous on the whole domain.   \end{abstract}

\section{Introduction}
\label{Ann:sec:Introduction}

In this paper we consider minimizers of variational problems that are motivated by nonlinear elasticity theory.   The functionals we wish to minimize are of the form 
\[I(u) = \int_{\om} W(\nabla u(x))\,dx,\]
where $\om \subset \R^2$ is a domain representing the reference configuration of an elastic material, $W: \R^{2 \times 2} \to [0,+\infty]$ its stored energy function and $u:\om \to \R^2$ a deformation.  One of the tenets of the theory is that the noninterpenetrability of matter is encoded by requiring that $\det \nabla u > 0$ a.e.\@in $\om$.   This is typically imposed by setting $W(F)=+\infty$ whenever the $2 \times 2$ matrix $F$ satisfies $\det F \leq 0$, so that any deformation having finite energy necesssarily satisfies $\det \nabla u >0$ a.e.     The main purpose of this paper is to  examine in particular the relationship between the positivity of the Jacobian $\det \nabla u$ and the uniqueness and regularity of two different kinds of stationary point associated with the energy functional $I(\cdot)$.  

The first kind of stationarity results in the so-called Energy-Momentum (EM) equations
\begin{equation}\label{em}\Div (\nabla u^T DW(\nabla u) - W(\nabla u)\mathbf{1})=0\end{equation} 
formally obtained from $I(\cdot)$ by setting $\partial_{\eps}\arrowvert_{\eps = 0}I(u^{\eps}) = 0$ in the case that $u^{\eps}(x) = u(x+\eps \varphi(x))$ and $\varphi$ is a smooth, compactly supported test function.   Conditions guaranteeing that \eqref{em} holds in a rigorous sense can be found in \cite{ Ball83:MinimizersEulerLagrange,Ball02:OpenProblemsInElasticity}.   
The second type of stationarity results formally in the Euler-Lagrange (EL) equations,
\begin{equation}\label{el}\Div DW(\nabla u) = 0,\end{equation}
whose derivation from $\partial_{\eps}\arrowvert_{\eps = 0}I(u+\eps \varphi) = 0$, when the latter exists, is well known.

The sorts of stationary point we consider fall into two broad classes: twists and shears.   Twist maps operate on an annulus $A=\{x \in \R^2: \ a < |x| < b\}$ and act as the identity on $\partial A$.  Shear maps are of the form $u(x) = x + \sigma(x) e$, where $e$ is a fixed unit vector and $\sigma$ a scalar field defined on some domain, which in this paper will typically be a square $Q:=[-1,1]^2$.   We study two types of functional in each of the twist and shear map classes: both are of the form $I(u) = \int_{\om} W(\nabla u(x))\,dx$ where the set $\om$ is either the annulus, $A$, or the square, $Q$, and $W$ is of the form
\begin{align}\label{www} W(F)=\frac{1}{2} |F|^2 + h(\det F)\end{align}
defined on $2 \times 2$ matrices $F$.  The function $h$ is either (i) of the kind that penalizes $\det F \to 0$ in the sense that $h=h_0$ and $h_0(s) \to +\infty$ as $s \to 0+$, $h_0(s)=+\infty$ for $s \leq 0$ and $h_0$ is convex where it is finite\footnote{See Section \ref{Ann:sec:WithVolumeCompression} for details of additional hypotheses imposed on $h_0$}, or (ii) of the form
\begin{displaymath} h_{\infty}(s) = \left\{\begin{array}{l l} 0 & \textrm{ if } \ s \geq 0 \\
+\infty & \textrm{if } \ s < 0. \end{array} \right.\end{displaymath}
Type (i) functions $h_0$ penalize compression to zero area, while type (ii) functions $h_{\infty}$ ensure that maps $u$ with finite energy obey $\det \nabla u \geq 0$ almost everywhere in $\om$.   

Thus there are effectively four permutations, and together they generate the range of behaviours summarised in the table below.  

\par  
\vspace{2mm}
\noindent \begin{tabular}{|p{1cm}|p{6cm}|p{6.5cm}|}\hline  & & \\
& \hspace{5mm} $W(F) = \frac{1}{2} |F|^2 + h_0(\det F)$ & \hspace{7mm}$W(F)=\frac{1}{2} |F|^2 + h_\infty (\det F)$ \\ 
 & & \\
\hline  
& & \\
Twist Maps & $\bullet$ infinitely many solutions of (EL) (see \cite{PostSivaloganathan97:HomotopyConditionsExistenceMultipleEquilibria}) \newline
$\bullet$ Jacobian bounded away from $0$ on $\bar{A}$ \newline
$\bullet$ solutions belong to the class $C^3(A)$ (see \cite{PostSivaloganathan97:HomotopyConditionsExistenceMultipleEquilibria}) & 
$\bullet$ infinitely many solutions of (EM) \newline
$\bullet$ Jacobian vanishes on set of positive measure \newline
$\bullet$ solutions are explicit and of class $C^1(A)$
\\ 
 & & \\
\hline 
 & & \\
Shear Maps & $\bullet$ solution of (EL) unique \footnote{Here, (EL) is a variational inequality: see Section \ref{ss1}.} \newline
$\bullet$ Singularities at boundary can form \footnote{This is provided the Jacobian is bounded away from $0$ and boundary conditions are of `mixed' type.}  & 
$\bullet$ solution of (EL) unique \newline
$\bullet$ Jacobian vanishes on set of positive measure \newline
$\bullet$ solution cannot be of class $C^1(Q)$ for appropriate boundary data \\ 
 & & \\
\hline 
\end{tabular}

\vspace{2mm}

The non-uniqueness of solutions to the Euler-Lagrange equations of elasticity problems with mixed boundary conditions is a well known phenomenon, such as in the buckling of a rod or beam.
However, for pure displacement boundary conditions things are not so clear.
Indeed, it is still an open question whether sufficiently smooth equilibrium solutions to pure displacement boundary-value problems for homogeneous bodies with strictly polyconvex stored energy function $W$ are unique if the domain $Ω$ is homeomorphic to a ball (see Problem 8, \cite{Ball02:OpenProblemsInElasticity}).   Much work has been done in this area: see \cite{SS16,B11,Taheri03:QuasiconvexityUniquenessStationaryPoints,KS84, Zhang91} and \cite{John72:UniquenessEquilibriumDisplacementBoundarySmallStrains}.    
F. John showed in \cite{John72:UniquenessEquilibriumDisplacementBoundarySmallStrains} that a twice continuously differentiable equilibrium of sufficiently small strain is unique.   In the same paper, the author formally suggested that multiple solutions to the Euler-Lagrange equations might be found among the twist maps of a two-dimensional annulus (cf.  Problem 8, \cite{Ball02:OpenProblemsInElasticity}).   Solutions of this kind were subsequently found by  Post and Sivaloganathan\footnote{These authors also extended their arguments to the torus in 3 dimensions.} in \cite{PostSivaloganathan97:HomotopyConditionsExistenceMultipleEquilibria} in the case that $h=h_0$, in the notation introduced above, and led to Francfort and Sivaloganathan's exploration of the case $h=h_\infty$ in \cite{FrancfortSivaloganathan02:ConservationLawsNecessaryConditions}.  When $h=h_{0}$, our contribution is to improve the regularity of the twist maps they found and to deduce that the Jacobian of each solution of the Euler-Lagrange equations is bounded away from zero, in contrast to the situation encountered when compression to zero area is not penalized, that is when $h=h_{\infty}$.  This is done by using techniques of Baumann, Owen and Phillips \cite{BaumanOwenPhilips91:MaximumPrinciplesAPrioriEstimates,BaumanOwenPhillips91:MaximalSmoothness} to show that auxiliary functions $d=\det \nabla u$ and  $z = \frac{1}{2}|\nabla u|^2 + f(\det \nabla u)$, where $f(d) = h_0'(d)d - h_0(d)$, are, respectively, monotonically increasing and decreasing along the radius of the annulus.
As an additional property, we also present a maximum principle for the function $\frac{\rho }{r}:=\frac{|u(x)|}{|x|}$, where $r=|x|$.   It remains an open question whether the global energy minimizers in this case are necessarily rotationally symmetric.

  In the case that $h=h_{\infty}$, we obtain infinitely many \emph{explicit}\footnote{These examples seem to be very rare in the literature.} rotationally symmetric solutions to the Energy-Momentum equations, which are parametrized by the number of times $N$, say, that the outer boundary $S_b:=\{ x \in \R^2: |x|=b\}$ of the annulus A is twisted around the inner boundary $S_a$ (using similar notation).   All these solutions share the property that an annular region $\{x \in \R^2: \ a \leq |x| \leq k\}$ around the inner boundary $S_a$ of $A$ is mapped onto $S_a$, thereby compressing that region to `zero area'.  This region, which we call the `hedgehog region' for reasons explained later in the paper, is where most of the twisting happens:  at most one quarter of the twist is performed outside the hedgehog region, regardless of the size of $N$.   See Section \ref{Ann:sec:WithVolumeCompression} for details.   It is interesting to note that our explicit solutions do not solve the Euler-Lagrange equations\footnote{To be precise, these take the form of a variational inequality.}, the proof of which relies on an observation of \cite{FrancfortSivaloganathan02:ConservationLawsNecessaryConditions}.  We also show that our equilibrium solutions are local minimizers in suitably restricted classes of twist maps:  see Proposition \ref{wolfgangAM} and Corollary \ref{suitsupport}.


In the context of shear maps, the results of Section \ref{She:sect} focus on the relationship between the regularity of global energy minimizers and the positivity of the Jacobian, among other things.  
Minimizing shear maps $\us$ are unique because the map
$\sigma \mapsto I(\us)$ is strictly convex as a functional and, as is explained in Section \ref{She:sect}, the class of admissible functions is convex as a set.  The former is obvious when $h=h_\infty$ and surprising when $h=h_0$: see Lemma \ref{lem:pen} for details.   
Using the same notation as above, we find a condition that characterizes the shear map minimizer of $I_{\infty}$ and which, in conjunction with a carefully chosen type of boundary condition, provides conditions under which the global shear map minimizer, $u_{\sigma,\infty}$, say, is not of class $C^1$.  The boundary condition, which can easily be generalized, ensures that $\det \nabla u_{\sigma,\infty}=0$ on a set of positive measure in $Q$, something it has in common with the twist solutions of Section \ref{twistwithhinfinity}.  

In the final part of the paper we prove that, under certain mixed boundary conditions, which again can be generalized, the shear map minimizer  $u_{\sigma,0}$, say, of $I_0$  is such that $\nabla u_{\sigma,0}$ is not continuous at the `corners' of $Q$.  This happens under the additional assumption that 
$\det \nabla u_{\sigma,0} \geq c > 0$ a.e.\@, which would normally be thought of as a regularizing condition, but which here seems to focus discontinuities in $\nabla u_{\sigma,0}$  at points on $\partial Q$ where the character of the boundary condition changes from mixed to traction-free.  The analysis relies on results from elliptic regularity theory that are applicable precisely because $\sigma \mapsto W(\nabla \us)$ is strongly convex.

\subsection{Notation}

We denote the $2 \times 2$ real matrices by $\mathbb{R}^{2 \times 2}$, and unless stated otherwise we sum over repeated indices.  The tensor product of two vectors  $a \in \mathbb{R}^{2}$ and $b \in \mathbb{R}^{2}$ is written $a \otimes b$; it is the $2 \times 2$ matrix whose $(i,j)$ entry is $a_{i}b_{j}$.   The inner product of two matrices $X,Y \in \mathbb{R}^{2 \times 2}$ is $X \cdot Y = \tr(X^{T}Y)$.   This obviously holds for vectors too.  For points $x=(r,\theta)$ in plane polar coordinates and belonging to a domain $\om \subset \R^2$, the gradient of $\varphi: \om \to \mathbb{R}^{2}$ is
\[ \nabla \varphi = \varphi_{,_r} \otimes e_{r}(\theta) + \frac{1}{r} \varphi_{,_\theta} \otimes e_{\tau}(\theta),\]
where $e_{r}(\theta) = (\cos \theta, \sin \theta)^{T}$ and
$e_{\tau}(\theta) = (-\sin \theta, \cos \theta)^{T}$.   Throughout the paper we write $\varphi_{,_{r}} = \partial_{r}\varphi$, $\varphi_{,_{\theta}} = \partial_{\theta}\varphi$ and $\varphi_{,_{\tau}} = \frac{1}{r}\partial_{\theta}\varphi$.      In this notation the formula
\[\det \nabla \varphi = J \varphi_{,_{r}} \cdot \varphi_{,_{\tau}}\]
holds, where $J$ is the $2 \times 2$ matrix corresponding to a rotation of $\frac{\pi}{2}$ radians in the plane, i.e.,
\[ J = \left( \begin{array}{l l} 0 & -1 \\ 1 & \phantom{-}0 \end{array}\right).\]
The two most useful properties of $J$ are that (i) $J^{T} = - J$, so that in particular $a \cdot Jb = - Ja \cdot b$ for any two $a,b \in \mathbb{R}^{2}$, and (ii) $\cof A = J^{T} A J$ for any $2 \times 2$ matrix $A$. We denote the identity matrix by $\1$.  Derivatives with respect to cartesian coordinates $x_i$ for $i=1,2$ will be usually be written $\varphi_{,_{x_i}}$, and occasionally $\partial_{x_i}\varphi$.

 A function $f: \mathbb{R}^{2 \times 2} \to \mathbb{R} \cup \{ +\infty\}$ is said to be polyconvex if there exists a convex function $\phi: \mathbb{R}^{2 \times 2} \times \mathbb{R}\to \mathbb{R} \cup \{+\infty\}$ such that  
\begin{align*}f(A) & = \phi(A,\det A) \end{align*}
for all $ 2 \times 2$ real matrices $A$.   The function space setting for all the problems we consider will be $W^{1,2}(\om;\R^2)$, which we will abbreviate to $W^{1,2}(\om)$ whenever it is unambiguous to do so.  As usual,  $\rightharpoonup$ represents weak convergence in both the Sobolev space $W^{1,2}(\om)$ and the Lebesgue space $L^2(\om)$.  Since $\om \subset \R^2$, the appropriate notion of boundary measure, as generated by the boundary integrals in Green's theorem, for example, is one-dimensional Hausdorff measure, which we write either as $d\mathcal{H}^1$ or, in the case of a circular boundary, $dS$. 

Other, standard notation includes $B(a,r)$ for the ball in $\mathbb{R}^{2}$ centred at $a$ with radius $r$ and $S_r$ for the circle centred at $0$ of radius $r$.  We write $A(p,q)$ for the annulus $B(0,q) \setminus \overline{B(0,p)}$, where $p < q$, and when it causes no confusion, we abbreviate $A(a,b)$ to $A$.

\section{Minimizers in the class of twist maps}\label{sectiontwo}
\label{Ann:sec:GeneralSetting}

We begin by recalling the technical setting of twist maps  first proposed in \cite{PostSivaloganathan97:HomotopyConditionsExistenceMultipleEquilibria}.  Let $A=\{x \in \R^2: a < |x|<b\}$ and set
\begin{align}
	\mathcal{A} = \{ u ∈ W^{1,2}(A) \st u = \id \ton \partial A \}.
	\label{Ann:Aadmissible}
\end{align}
Following \cite[Section 2]{PostSivaloganathan97:HomotopyConditionsExistenceMultipleEquilibria}, one now selects subclasses of $\mathcal{A}$ by means of the winding number.  Formally, for each integer $N$ we restrict attention to maps $u: A \to \R^2$ which rotate the outer boundary $\{x \in \R^2: \ |x| = b\}$ $N$ times relative to the inner boundary $\{x \in \R^2: \ |x| = a\}$.  
More precisely, changing to polar coordinates and applying the ACL property of Sobolev functions, it is the case that for a.e.\@ $\theta \in [0,2\pi]$ the curve 
\[ \gamma_{\theta}:=\left\{\frac{u(r,\theta)}{|u(r,\theta)|}: \ a \leq r \leq b \right\}\]
is closed and continuous.  The winding number for such curves is defined by approximation using $C^1$ curves in the plane.   We recall that the winding number of a closed $C^1$ curve in the plane, i.e. $γ:[a,b] → ℝ^2 $ with $γ(a) = γ(b)$ and $γ(r) = (x(r), y(r))$, is defined by
\vspace{2mm}
\begin{align}
	\mathop{\textnormal{wind}}\# γ = \frac{1}{2π} ∫_a^b \frac{x(r) y'(r) - x'(r) y(r)}{x^2(r) + y^2(r)} \d r.
	\label{Ann:windingNumber}
\end{align}

For each integer $N$ let  
 \begin{align}
	\mathcal{A}_N = \{ u ∈ \mathcal{A} \st \mathop{\textnormal{wind}}\# γ_θ = N \text{ for a.e. } θ∈[0,2π] \}.
	\label{Ann:ANadmissible}
\end{align}
By \cite[Lemma 2.7]{PostSivaloganathan97:HomotopyConditionsExistenceMultipleEquilibria} each class $\mathcal{A}_N$ is closed with respect to weak convergence in $W^{1,2}(A)$.   The existence of a minimizer of $I(u)=\int_{A} W(\nabla u)\, dx$ then follows easily by applying the direct method of the calculus of variations.   We will apply this procedure both in the case that compression to zero area is penalized and when it is not, corresponding respectively to the choice $h=h_0$ and $h=h_{\infty}$ in the stored-energy function $W$.   We turn first to the case $h=h_\infty$.  
 
\subsection{The case $h=h_{\infty}$: twist minimizers without area compression energy}\label{twistwithhinfinity}

The problem we consider here was raised in Francfort and Sivaloganathan 
\cite{FrancfortSivaloganathan02:ConservationLawsNecessaryConditions} and is illustrative of the case where the Euler-Lagrange equations are not satisfied by minimizers: see Remark \ref{rk:ELnotsatFS} below.   Using the framework of \cite{PostSivaloganathan97:HomotopyConditionsExistenceMultipleEquilibria}, our approach is to seek solutions of the Energy-Momentum equations for the functional
\[ I_{\infty}(u) = \int_{A} \frac{1}{2}|\nabla u|^2+ h_{\infty}(\det \nabla u) \,dx \]
in the class
 \begin{align}
	\mathcal{A}_N = \{ u ∈ \mathcal{A} \st \mathop{\textnormal{wind}}\# γ_θ = N \text{ for a.e. } θ∈[0,2π] \},
	\end{align}
where the class $\sca$ is given by \eqref{Ann:Aadmissible}.    This is clearly equivalent to minimizing a Dirichlet energy 
\begin{align}
	D(u) := ∫_A |\nabla u|^2 \, dx
	\label{Ann:I0AnnulusFS}
\end{align}
on the set
\begin{align}
	\mathcal{\tilde{A}}_N = \{ u \in \tilde{\sca} \st \mathop{\textnormal{wind}}\# γ_θ = N \text{ for a.e. } θ∈[0,2π] \},
	\label{Ann:AN0admissible}
\end{align}
where 
\begin{align}
	\mathcal{\tilde{A}} = \{ u ∈ W^{1,2}(A) \st u = \id{ } \ton ∂A \text{ and } \det ∇u ≥ 0 \text{ a.e. in } A \}.
	\label{Ann:A0admissible}
\end{align}

\begin{proposition}\label{sf1} Let $I_{\infty}$ and $\sca_{N}$ be as above.  Then there is a minimizer of $I_{\infty}$ in $\sca_{N}$. \end{proposition}
\begin{proof} We apply the direct method of the calculus of variations to the  
formulation of the problem in terms of the Dirichlet integral $D(u)$.    Note that $\tilde{\sca}_{N}$ contains the map
\[ U(x)= r \left(\cos\left(\theta+ 2N \pi\left(\frac{r-a}{b-a}\right)\right), \, \sin\left( \theta + 2N \pi\left(\frac{r-a}{b-a}\right)\right)\right), \]
where $x = r (\cos \theta,\sin \theta)$, so that $\tilde{\sca}_{N}$ is in particular nonempty.  To show that it is weakly closed we appeal first to \cite[Lemma 2.7]{PostSivaloganathan97:HomotopyConditionsExistenceMultipleEquilibria} to ensure that the weak limit $u$, say, in $W^{1,2}(A)$ of any sequence $u^{(j)}$ in $\tilde{\sca}_{N}$ obeys the winding number constraint and boundary conditions.  Moreover, from \cite[Corollary 1.2]{Muller:LlogL}, it follows that $\det \nabla u \geq 0$ a.e. in $A$ when $\det \nabla u^{(j)} \geq 0$ a.e.\@ holds for all $j$ and $\nabla u^{(j)} \rightharpoonup \nabla u$ in $L^{2}(A)$.  Hence $\tilde{\sca}_{N}$ is weakly closed.  A straightforward argument using the convexity of the Dirichlet energy implies that $D(\cdot)$ is sequentially weakly lower semicontinuous, from which the existence of a minimizer follows.   \end{proof}

\begin{remark}\label{rk:ELnotsatFS}\emph{
We expect that the minimizer $u^N$ of $I_{\infty}$ in $\tilde{\mathcal{A}}_N$ for $N≠0$ to be degenerate in the sense that $\det \nabla u^{N}$ cannot be bounded away from $0$.   This is because if there exists $c>0$ such that $\det \nabla u^N \geq c$ in $A$ then the Euler-Lagrange equations for $I_{\infty}$ are equivalent to
\begin{align}
	\left\{
		\begin{alignedat}{3}
			Δu = 0 	&\quad\tin A \\
			u = \id &\quad\ton ∂A,
		\end{alignedat}
	\right. 
	\label{Ann:Laplace}
\end{align}
which, by standard theory, has the unique solution $u = \id$, and which does not obey the winding number condition ($u=\id$ clearly has winding number zero).   One way in which $\det \nabla u^{N}$ could fail to be a.e.\@ bounded away from $0$ is for it to vanish on a set of positive measure in $A$:  this is certainly the case for the symmetric minimizers of which we give details later.}  \end{remark}

It is straightforward to check that the energy momentum equations associated with $I_{\infty}$ are, in a distributional sense, 
\begin{align}
	\left\{
		\begin{alignedat}{3}
			\div \left( \frac{1}{2}|\nabla u|^2 \1 -  \nabla u^T \nabla u \right) = 0 & \quad\tin A \\
			u = \id	& \quad\ton ∂A. 
		\end{alignedat}
	\right. 
	\label{Ann:energyMomentumAnnulus0}
\end{align}
We seek a rotationally symmetric solution of this system, i.e. a solution from the set
\begin{align}
	\tilde{\mathcal{A}}_{N,\,\ssym} = \{ u \in  \tilde{\mathcal{A}}_N \st u(x) = Q^Tu(Qx) \text{ for all } Q\in SO(2) \}.
	\label{Ann:A0Nsym}
\end{align}
That such a solution exists follows from the same argument used in the proof of Proposition \ref{sf1}.  
Rotationally symmetric solutions can be represented in polar coordinates as
\begin{align}
	u(r,θ) = \rho (r) e_r( θ + ψ(r))
	\label{Ann:uRotSym}
\end{align}
where $e_r(θ) = (\cos θ, \sin θ)$. 
For brevity we shall henceforth write $e_r$ for $e_r(θ)$ and $\widetilde{e_r}$ for $e_r(θ+ψ(r))$.
Similarly, we define $e_θ(θ) = (-\sin θ, \cos θ)$ and use the abbreviations $e_θ$ and $\widetilde{e_θ}$ analogously.
We call $\rho $ the radial map and $ψ$ the angular map.  In this notation, we have the following result.

\begin{lemma}
	Let $N\in \mathbb{N}$.
	Then the radial map $\rho $ of a minimizer of $I_{\infty}$ in $\tilde{\mathcal{A}}_{N,\,\ssym}$ is differentiable and satisfies the ODE
	\begin{align}
		\dot{\rho } = \frac{1}{r} \sqrt{ \rho ^2 - \frac{\omega ^2}{\rho ^2} - a^2 + \frac{\omega ^2}{a^2} } \\
		\rho (a) = a, \quad \rho (b) = b.
		\label{Ann:DiffEqRho0}
	\end{align}
	for some $\omega \in (0,+\infty )$.
	Furthermore, the angular map $ψ$ is differentiable with
	\begin{align}
		\dot{ψ} = \frac{\omega }{R\rho ^2}
		\label{Ann:angularMap}
	\end{align}
	and $ψ(a) = 0$ and $ψ(b) = 2πN$.
	\label{th:Ann:DiffEqRho0}
\end{lemma}
\begin{proof}
	To prove this we test the weak form of \eqref{Ann:energyMomentumAnnulus0} with a rotationally symmetric test function $ϕ$.
	We can express $ϕ$ as 
	\begin{align}
		ϕ(r,θ) = \hat{\rho }(r) e_r + \hat{q}(r) e_θ.
	\end{align}
	with $\hat{\rho },\hat{q}\in C^\infty _c( (a,b) )$.
	Furthermore 
	\begin{align}
		\nabla u &= \dot{\rho } \etR \otimes e_r + \rho \dot{ψ}\etT \otimes e_r + \frac{\rho }{r}\etT \otimes e_θ,		\label{Ann:gradU} \\
		\nabla ϕ &= \dot{\hat{\rho }} e_r \otimes e_r + \dot{\hat{q}}e_θ \otimes e_r + \frac{1}{r}\left[ \hat{\rho }e_θ \otimes e_θ - \hat{q}e_r \otimes e_θ \right], 	\label{Ann:gradPhi} 
		\intertext{and}
		|\nabla u|^2 &= \dot{\rho }^2 + \rho ^2\dot{ψ}^2 + \frac{\rho ^2}{r^2}.
	\end{align}
	Therefore,
	\begin{align}
		\frac{1}{2} |\nabla u|^2 I - \nabla u^T \nabla u &=  \frac{1}{2}\left(\dot{\rho }^2 + \rho ^2\dot{ψ}^2 + \frac{\rho ^2}{r^2} \right) I - \left( (\dot{\rho }^2 + \rho ^2\dot{ψ}^2) e_r \otimes e_r \phantom{\frac{\rho ^2\dot{ψ}}{r}} \right.\notag\\
		&\quad + \left. \frac{\rho ^2\dot{ψ}}{r}( e_r \otimes e_θ + e_θ \otimes e_r)  + \frac{\rho ^2}{r^2}e_θ \otimes e_θ \right),
		\label{Ann:energyMomentumTensor0}
	\end{align}
	so that
	\begin{align}
		0 	&= ∫_A \left[ \frac{1}{2}|\nabla u|^2 I - \nabla u^T \nabla u \right] \cdot \nabla ϕ \, dx \notag\\
		&= 2π∫_a^b \frac{r}{2} \left(\dot{\rho }^2 + \rho ^2\dot{ψ}^2 + \frac{\rho ^2}{r^2}\right) \left(\dot{\hat{\rho }} + \frac{\hat{\rho }}{r} \right) \notag\\
		&\qquad\qquad - r \left( (\dot{\rho }^2 + \rho ^2\dot{ψ}^2) \dot{\hat{\rho }} + \frac{\rho ^2\dot{ψ}}{r}\left( \dot{\hat{q}} - \frac{\hat{q}}{r} \right) + \frac{\rho ^2}{r^2}\frac{\hat{\rho }}{r} \right) \, dr \notag\\
		&= 2π ∫_a^b \frac{r}{2} \left( \dot{\rho }^2 + \rho ^2\dot{ψ}^2 - \frac{\rho ^2}{r^2} \right)\left( \frac{\hat{\rho }}{r} - \dot{\hat{\rho }} \right) + r \frac{\rho ^2\dot{ψ}}{r}\left( \frac{\hat{q}}{r} - \dot{\hat{q}} \right) \, dr \notag\\
		&= - 2π ∫_a^b \frac{r^2}{2} \left( \dot{\rho }^2 + \rho ^2\dot{ψ}^2 - \frac{\rho ^2}{r^2} \right) \left( \frac{\hat{\rho }}{r} \right)^{\Cdot} + r\rho ^2\dot{ψ}\left( \frac{\hat{q}}{r} \right)^{\Cdot} \, dr.
		\label{Ann:energyMomentum}
	\end{align}
	Since $\hat{\rho }$ and $\hat{q}$ are arbitrary this implies that there exist constants $c$ and $\omega $ s.t.
	\begin{align}
		r^2\left( \dot{\rho }^2 + \rho ^2\dot{ψ}^2 - \frac{\rho ^2}{r^2} \right) &= c	\quad \tin (a,b) \label{Ann:ODE1}
		\intertext{and}
		r\rho ^2\dot{ψ}  &= \omega  \quad \tin (a,b).
		\label{Ann:ODE2}
	\end{align}
	Furthermore, since $∫_A |\nabla u|^2 \, dx < \infty $, it follows that $\rho \in W^{1,2}( (a,b) )$, which in turn yields $\rho \in C([a,b])$.
	Therefore $\dot{ψ} = \frac{\omega }{r\rho ^2}\in C([a,b])$ as well.
	That $\omega >0$ simply follows from the fact that, by \eqref{Ann:ODE2}, $ψ$ is a monotonic function and we want to achieve a positive winding number, i.e. $ψ(b) = 2πN>0$ and $ψ(a) = 0$.
	Substituting $\dot{ψ}$ back into \eqref{Ann:ODE1} we obtain
	\begin{align}
		r^2\dot{\rho }^2 + \frac{\omega ^2}{\rho ^2} - \rho ^2 = c
		\label{Ann:ODE_rho}
	\end{align}
	which also implies that the weak derivative $\dot{\rho }$ is continuous and is therefore the classical derivative. 
	Since $\det \nabla u = \frac{\rho \dot{\rho }}{r} \geq  0$, we find that $\rho ^2$ is monotonically increasing. 
	Therefore $\rho  \geq  a > 0$ which in turn implies $\dot{\rho } \geq  0$.
	Hence we can solve for $\dot{\rho }$ in \eqref{Ann:ODE_rho} to obtain \eqref{Ann:DiffEqRho0}.
	\par
	Now we want to prove that $c = -a^2 + \frac{\omega ^2}{a^2}$.
	In view of \eqref{Ann:ODE_rho}, this is equivalent to showing that $\dot{\rho }(a)$ has to be zero.   This is done in two steps:  first we show that if $\dot{\rho}$ vanishes then it can only do so at $r=a$, and then we prove that $\dot{\rho}(a)>0$ is impossible, which, since $\dot{\rho}$ is nonnegative, leaves only the possibility that $\dot{\rho}(a) = 0$.

Assume for a contradiction that there is a point $\bar{r}\in (a,b]$ s.t. $\dot{\rho }(\bar{r}) = 0$ and $\dot{\rho }(r)>0$ for $r\in (\bar{r}-δ,\bar{r})$ for some $δ>0$, meaning that we suppose $\dot{\rho }$ has a zero at the rightmost point of an interval where it is strictly positive.   Let $z(r)=f(\rho(r))$ where $f(\rho) = \rho^2 - \frac{\omega ^2}{\rho^2} + c$ and note that, by \eqref{Ann:ODE_rho}, $z(r) > 0$ if $\bar{r} -\delta < r < \bar{r}$ and $z(\bar{r}) = 0$.  On the other hand, a short calculation shows that $\dot{z}(r) >0$ on $(\bar{r}-\delta,\bar{r})$, and hence that $z(r) < 0$ on the same interval, a contradiction.  Thus the only possibility is that $\dot{\rho}(a) = 0$ if $\dot{\rho}$ vanishes at all.

	Now assume for a contradiction that $\dot{\rho }(a)>0$.  Then, since $\dot{\rho }\in C([a,b])$ and by the reasoning above, it is bounded away from zero on the whole of $[a,b]$, i.e. $\dot{\rho }\geq \epsilon >0$ for some $\epsilon >0$.
	But in this case, by Remark \ref{rk:ELnotsatFS}, $u$ solves the Euler-Lagrange equations
	\begin{align}
		\left\{
			\begin{alignedat}{3}
				\Delta  u = 0 	&\quad	\tin A \\
				u = \id &\quad	\ton ∂A, 
			\end{alignedat}
		\right. 
	\end{align}
	which admit only the identity as a solution,  corresponding to $N=0$. This contradicts the winding number condition in force on $\tilde{A}_{N}$.  Hence $\dot{\rho}(a)=0$.  
\end{proof}

In short, the preceding lemma implies that we can reduce the energy-momentum equations for $\rho $ and $ψ$ to an ODE in $\rho $ with the initial condition $\rho (a) = a$.
It might seem strange that there is only 
 one parameter $\omega$ left to fit both the boundary condition $\rho (b) = b$ and to ensure that $ψ(b) = 2πN$.
However, the lack of Lipschitz continuity of the right hand side of \eqref{rk:ELnotsatFS} means that there are infinitely many solutions for each $\omega $ that differ qualitatively only by the point $k \in (a,b)$ where $\dot{\rho }$ first departs from zero, and which is therefore an additional, hidden parameter.
A rather unusual result is that this system of ODEs, and therefore the Energy-Momentum equations from which they are derived, has an \emph{explicit} solution.
\begin{theorem}
	Let $N\in \mathbb{N}$.
	Then there exist $\omega >0$ and $ k \in [a,b)$ s.t.
	\begin{align}
		\rho (r) = 
		\begin{cases}
			a, 	& r\in [a,k] \\
			\frac{1}{2} \left( \left( a^2 + \frac{\omega ^2}{a^2} \right)\frac{r^2}{k^2} + \left( a^2 + \frac{\omega ^2}{a^2} \right) \frac{k^2}{r^2} + 2\left( a^2 - \frac{\omega ^2}{a^2} \right) \right)^{\frac{1}{2}}, 	& r\in (k,b]
		\end{cases}
		\label{Ann:solution0}
	\end{align}
	is a solution to the ODE
	\begin{align*}
		r^2\dot{\rho }^2 + \frac{\omega ^2}{\rho ^2} - \rho ^2 = -a^2 +\frac{\omega^2}{a^2}
	\end{align*}
	derived in Lemma~\ref{th:Ann:DiffEqRho0}.
	Furthermore, $\omega $ and $k \in (a,b)$ are uniquely determined.
	The corresponding angular map is
	\begin{align}
		ψ(r) = 
		\begin{cases}
			\frac{\omega }{a^2} \ln\left( \frac{r}{a} \right), 	& r\in [a,k] \\
			\frac{\omega }{a^2} \ln\left( \frac{k}{a} \right) + \arctan\left( \frac{1}{2\omega } \left[ \left( a^2 + \frac{\omega ^2}{a^2} \right)\frac{r^2}{k^2} + a^2 - \frac{\omega ^2}{a^2} \right] \right) - \arctan\left( \frac{a^2}{\omega } \right),		& r\in (k,b].
		\end{cases}
		\label{Ann:solutionAngularMap}
	\end{align}
	\label{th:Annulus:SolODE}
\end{theorem}
\begin{proof}
	It is easy to directly verify that the map $\rho $ given above solves the ODE.
	The existence of $\omega$ and $k$ is ensured by the existence of the minimizer.
	It remains to check that the boundary conditions $\rho (b) = b$ and $ψ(b) = 2πN$ are met.
	Now, the condition $\rho (b) = b$ fixes $\omega >0$ as a function of $k$:
	\begin{align}
		\omega ^2 = \frac{4b^4k^2a^2 - a^4(b^2+k^2)^2}{(b^2-k^2)^2}.
		\label{Ann:omegaOfh}
	\end{align}
	Inserting this into \eqref{Ann:solutionAngularMap}, we find that $ψ(b)$ is then a continuous function of $k$.  Let us briefly write $\psi(b;k)$ to make the dependence on the parameter $k$ explicit.  We seek $k \in (a,b)$ such that $\psi(b;k)=2 \pi N$.   	It can easily be checked that 
$k \mapsto \psi(b;k)$ has a pole at $k=b$, i.e. $ψ(b;k) → \infty $ as $k→b$, and that $\psi(b;k)$
 is monotonically increasing in $k$ for $k < b$.  Hence there is a unique $k$ in $(a,b)$ such that $\psi(b,k)=2 \pi N$.   We also note that since
	\begin{align*}\frac{1}{2\omega } \left[ \left( a^2 + \frac{\omega ^2}{a^2} \right)\frac{b^2}{k^2} + a^2 - \frac{\omega ^2}{a^2} \right] & > \frac{a^2}{\omega } > 0,\end{align*}
	less than a quarter of a twist is performed in the image of the annulus $A(k,b)$, that is $\psi(b) - \psi(k) < \frac{π}{2}$.
\end{proof}

The solution obtained for $N=1$ is sketched in Fig.~\ref{fig:Ann:SolutionN1}.
We define the set $H = \{ x\in ℝ^2 \st a\leq  |x| \leq  h\} ⊆ A$ to be the region that is mapped onto the circle $S_a$, and refer to it as the \emph{hedgehog} region.  The reason for this name is that the map $x\mapsto \frac{x}{|x|}$ is commonly referred to as the hegdehog map, and in the region $H$ the solution corresponds to a scaled version of this map with added twist.
\begin{figure}[htpb]
	\begin{center}
		\begin{tikzpicture}
			\draw[line width=1.0pt] (0,0) circle (2.5cm);
			\foreach \r in { {(2.5-1.175)/4+1.175},{2*(2.5-1.175)/4+1.175},{3*(2.5-1.175)/4 +1.175},2.5}
			\draw[style=dashed] (0,0) circle (\r);
			\foreach \thN in {1,2,...,7}
			\draw[style=dashed] ({0.5*cos(\thN*2*pi/8 r)},{0.5*sin(\thN*2*pi/8 r)}) -- ({2.5*cos(\thN*2*pi/8 r)},{2.5*sin(\thN*2*pi/8 r)});
			\draw[even odd rule,fill=gray,opacity=0.2] (0,0) circle (1.175cm) (0,0) circle (0.5cm);
			\draw[line width=1.0pt,color=red] (0.5,0) -- (2.5,0);
			\draw[line width=1.0pt] (0,0) circle (0.5cm);
			\foreach \r in { 0.5,{(1.175-0.5)/3+0.5},{2*(1.175-0.5)/3+0.5},1.175 }
			\draw[fill=black] ({\r},0) circle (0.05);
			\draw[->] ({2.5+0.2},0) .. controls (3.5,0.1) .. ({4.5-0.2},0);
			\node[above] at (3.5,0.1) {$u$};
			\draw[line width=1.0pt] (7,0) circle (2.5cm);
			\draw[line width=1.0pt] (7,0) circle (0.5cm);
			\def\h{1.175}
			\def\w{1.456}
			\def\c{-8.230}
			\def\a{12.052}
			\foreach \thN in {0,1,...,7}
			\draw[style=dashed,domain=\h:2.5,samples=20,variable=\x] 
			plot (	{7+0.5*sqrt( (4*(\w)^2+(\c)^2)/\a*(\x)^2 + \a/(\x)^2 + 2*\c)*cos(360/8*\thN+\w/0.25*ln(\h/0.5)r+atan(1/(2*\w)*((4*(\w)^2+(\c)^2)/\a*(\x)^2 + \c))-atan(1/(2*\w)*((4*(\w)^2+(\c)^2)/\a*(\h)^2 + \c)))},
			{  0.5*sqrt( (4*(\w)^2+(\c)^2)/\a*(\x)^2 + \a/(\x)^2 + 2*\c)*sin(360/8*\thN+\w/0.25*ln(\h/0.5)r+atan(1/(2*\w)*((4*(\w)^2+(\c)^2)/\a*(\x)^2 + \c))-atan(1/(2*\w)*((4*(\w)^2+(\c)^2)/\a*(\h)^2 + \c)))});
			\foreach \r in { {(2.5-1.175)/4+1.175},{2*(2.5-1.175)/4+1.175},{3*(2.5-1.175)/4+1.175},2.5}
			\draw[style=dashed] (7,0) circle ({0.5*sqrt( (4*(\w)^2+(\c)^2)/\a*pow(\r,2) + \a/pow(\r,2) + 2*\c)});
			\draw[thick,color=red,domain=0.5:\h,samples=50]
			plot ({7+0.5*cos(\w/0.25*ln(\x/0.5)r)}, {0.5*sin(\w/0.25*ln(\x/0.5)r)});
			\draw[thick,color=red,domain=\h:2.5,samples=20]
			plot (	{7+0.5*sqrt( (4*(\w)^2+(\c)^2)/\a*(\x)^2 + \a/(\x)^2 + 2*\c)*cos(\w/0.25*ln(\h/0.5)r+atan(1/(2*\w)*((4*(\w)^2+(\c)^2)/\a*(\x)^2 + \c))-atan(1/(2*\w)*((4*(\w)^2+(\c)^2)/\a*(\h)^2 + \c)))},
			{  0.5*sqrt( (4*(\w)^2+(\c)^2)/\a*(\x)^2 + \a/(\x)^2 + 2*\c)*sin(\w/0.25*ln(\h/0.5)r+atan(1/(2*\w)*((4*(\w)^2+(\c)^2)/\a*(\x)^2 + \c))-atan(1/(2*\w)*((4*(\w)^2+(\c)^2)/\a*(\h)^2 + \c)))});
			\foreach \r in { 0.5,0.725,0.95,1.175 }
			\draw[fill=black] ({7+0.5*cos(\w/0.25*ln(\r/0.5)r)}, {0.5*sin(\w/0.25*ln(\r/0.5)r)}) circle (0.05);
		\end{tikzpicture}
	\end{center}
	\caption{Sketch of the solution for $N=1$.  The grey region $H$ - the hedgehog region - is mapped onto the circle $S_a$.}
	\label{fig:Ann:SolutionN1}
\end{figure}
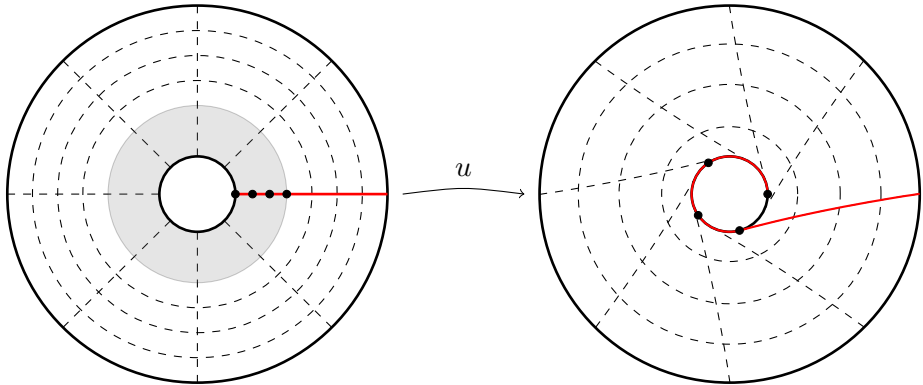
\par
So far we  have only considered rotationally symmetric maps and, for each $N\in \mathbb{N}$, we have found a unique minimizer $u_{*}^{N}$, say, in $\tilde{\mathcal{A}}_{N,\,\ssym}$, where the asterisk subscript refers to the rotational symmetry of the map. 
At the moment it is not clear whether $u_{*}^{N}$ is also a global minimizer of $I_{\infty}$ in the full class $\tilde{\mathcal{A}}_{N}$.  A natural first step towards obtaining such a result would be to prove that the global minimizer of $I_{\infty}$ in $\tilde{\sca}_{N}$ is rotationally symmetric, but we are currently unable to do this.   What we can say, however, is that $u_{*}^{N}$ is an energy minimizer with respect to variations belonging to the larger class $\tilde{\sca}_{N}$ and which obey certain conditions.  Before stating these conditions, a short technical lemma is required.

\begin{lemma}\label{brahms1} Let $\varphi \in C^2_{0}(A)$.  Then for each $R \in (a,b)$ 
\begin{align}\label{identity1} \int_{A(a,R)} \det \nabla \varphi \, dx  = \frac{1}{2} \int_{S_R} J\varphi \cdot \varphi_{,_{\tau}} \,dS.\end{align}

\end{lemma}
\begin{proof}  In the following we make use of the identity $\det \nabla \varphi = J \varphi_{,_{r}} \cdot \varphi_{,_{\tau}}$, where $\phi_{,_{\tau}} = \frac{1}{r} \frac{\partial \varphi}{\partial \theta}$ and $(r,\theta)$ are standard polar coordinates in two dimensions.   It may also help to recall at this point that the $2 \times 2$ matrix $J$ is antisymmetric.  For any $R$ in the interval $(a,b)$ 
\begin{align*} \int_{A(a,R)} \det \nabla \varphi \,dx & = \int_{a}^{R}\int_{0}^{2\pi}  J \varphi_{,_{r}}\cdot \varphi_{,_{\theta}}  \,d \theta \,dr \\
& =- \int_{a}^{R} \int_{0}^{2\pi}  (J \varphi_{,_{r}})_{,_\theta}\cdot \varphi \,d \theta \,dr \\
& =- \int_{0}^{2\pi} \int_{a}^{R}   (J \varphi_{,_{\theta}})_{,_{r}}\cdot \varphi \,dr \,d \theta  \\
& =- \int_{0}^{2\pi} J \varphi_{,_{\theta}}(R,\theta)\cdot \varphi(R,\theta) \,d \theta  +\int_{0}^{2\pi}\int_{a}^{R} J\varphi_{,_{\theta}} \cdot \varphi_{,_{r}} \,dr \,d \theta \\
& =- \int_{S_{R}} J \varphi_{,_{\tau}}\cdot \varphi  \,dS  - \int_{A(a,R)} \varphi_{,_{\tau}} \cdot J \varphi_{,_{r}} \,dx.\\
\end{align*}
We recognise the integrand of the rightmost term in the final line as $\det \nabla \varphi$, whereupon  \eqref{identity1} follows by rearranging the terms and observing that $-J \varphi_{,_{\tau}} \cdot \varphi = \varphi_{,{\tau}} \cdot J \varphi$.  
\end{proof}

\begin{proposition}\label{wolfgangAM}
	Let $N\in \mathbb{N}$ and let $u_{*}^{N}$ minimize $I_{\infty}$ in $\tilde{\mathcal{A}}_{N,\,\ssym}$.  
		\begin{itemize}
		\item[(i)] Let $T^+ \tilde{\sca}_{N}=\{\varphi \in W^{1,2}_{0}(A): \ \ u_{*}^{N}+\eps \varphi \in \tilde{A}_{N} \ \textrm{ for all sufficiently small} \ \eps>0\}$.   
		Then for each $\varphi$ in $T^+ \tilde{\sca}_{N}$ 
\begin{align}\label{weakmin} I_{\infty}(u_{*}^{N}+\eps \varphi) & \geq I_{\infty}(u_{*}^{N})\end{align}
	for all sufficiently small and positive $\eps$.
	
	\item[(ii)]  Let $v \in \tilde{\sca}_{N}$ be such that $\varphi:=v - u_{*}^{N}$
	satisfies
	\begin{align*} \int_{H} |\nabla \varphi|^2 + 2 \left(1+\frac{\omega^2}{a^2}\right)\left(\frac{1}{k} - \frac{1}{r}\right) \det \nabla \varphi \,dx & \geq 0.  \end{align*}
	Then \[ I_{\infty}(v) \geq I_{\infty}(u_{*}^{N}).\]
	\end{itemize}

\end{proposition}

\begin{proof}
For brevity, let $u_{*}^{N}=u$ in the following, and recall that $D(u) = \int_{A}|\nabla u|^2 \,dx$.    The proof of parts (i) and (ii) have a common beginning which relate the quantity $\int_{A} \nabla u \cdot \nabla \varphi \, dx$ to terms involving $\cof \nabla u \cdot \nabla \varphi$.  The former term is clearly of importance when one considers the expansion 
\begin{align}\label{hubertparry}D(u+\varphi) & = D(u) + 2 \langle \nabla u , \nabla \varphi \rangle + D(\varphi)
\end{align}
and where one looks for conditions guaranteeing that (at least) one of $\langle \nabla u , \nabla \varphi \rangle$ and $D(\varphi)+ \langle \nabla u , \nabla \varphi \rangle$ is nonnegative.   Here $\langle \cdot , \cdot  \rangle$ is the $L^2(A)$ inner product.

First observe that since $u$ is smooth on $H$ and $A \setminus H$ and its first derivatives are continuous across the boundary $S_h$, Green's theorem implies that
	\begin{align}
		∫_A \nabla u \cdot \nabla \varphi  \, dx = -∫_H \Delta u \cdot \varphi  \, dx	
		\label{Ann:partialIntLaplace}
	\end{align}
	Notice that, since $u$ is harmonic on $A\setminus H$, the domain of integration of the right-hand side is the set $H$.
	Next, the specific form of the solution $u$ implies that $\Delta u = -\frac{a}{r^2}\left( \frac{\omega ^2}{a^2}+1 \right)\etR$, so that
	\begin{align}
		∫_A \nabla u \cdot \nabla \varphi  \, dx = a\left( \frac{\omega ^2}{a^2} + 1 \right) ∫_H \frac{1}{r^2} \etR \cdot \varphi   \,dx.
		\label{Ann:DerivativeIin_etr}
	\end{align}
	Now, using the same notation as in the previous lemma, we can integrate $\cof \nabla u \cdot \nabla \varphi $ on $A(a,R)$ for each fixed $R \in (a,b)$ to obtain 
\begin{align}
	∫_{A(a,R)} \cof \nabla u \cdot \nabla \varphi  \, dx = ∫_{S_R} (\cof \nabla u)n \cdot \varphi  \, dS = a ∫_{S_R} \etR \cdot \varphi  \, dS.
		\label{Ann:cofactorPartialIntegration}
	\end{align}	
	Here, the specific form of the solution $u$ has been used again: to be precise, one uses \eqref{Ann:gradU} to calculate $\cof \nabla u = \frac{\rho}{r} \etR$, which together with Piola's identity $\Div (\cof \nabla u) = 0$ and Green's theorem yields the stated expression.   The point we exploit below is that the quantity $\etR \cdot \varphi$ appears in both \eqref{Ann:DerivativeIin_etr} and \eqref{Ann:cofactorPartialIntegration}, enabling us to control the term $\langle \nabla u, \nabla \varphi \rangle$ using information about $\cof \nabla u \cdot \nabla \varphi$.  
		
\vspace{1mm}
	
\noindent{\textbf{Proof of (i)}} Let $\varphi$ belong to $T^+ \tilde{\sca}_{N}$.  Then for all sufficiently small $\eps>0$ 
\[ \det \nabla u + \eps\cof \nabla u \cdot \nabla \varphi + \eps^2 \det \nabla \varphi \geq 0 \]
a.e.\@ in $A$, and since $\det \nabla u =0 $ on $H$ it is in particular true that 
\[ \eps\cof \nabla u \cdot \nabla \varphi + \eps^2 \det \nabla \varphi \geq 0 \]
on $H$.  Dividing by $\eps>0$ and letting $\eps \to 0$ yields $\cof \nabla u \cdot \nabla \varphi \geq 0$ pointwise a.e.\@ in $H$.  From this and \eqref{Ann:cofactorPartialIntegration} it follows that 
\[ a∫_{S_R} \etR \cdot \varphi  \, dS \geq 0 \]
for $R \in (a,h)$.  Replacing $R$ by $r$, multiplying the latter inequality by 
\[ \zeta(r):= \left( \frac{\omega ^2}{a^2} + 1 \right)\frac{1}{r^2} \]
and integrating with respect to $r$ over $(a,k)$ implies, by \eqref{Ann:DerivativeIin_etr}, that $\langle \nabla u , \nabla \varphi \rangle \geq 0$.  Hence, by replacing $\varphi$ with $\eps \varphi$ in \eqref{hubertparry}, we must have $D(u+\eps \varphi) \geq D(u)$ for all sufficiently small $\eps > 0$.  It follows that \eqref{weakmin} must hold, which concludes the proof of part (i).

\vspace{1mm} 

\noindent{\textbf{Proof of (ii)}} Let $v \in \tilde{\sca}_{N}$ be admissible and let $\varphi=v-u$.  Since $v$ is admissible and $\det \nabla u = 0$ a.e.\@ on $H$, we can argue as above that
\[ \cof \nabla u \cdot \nabla \varphi + \det \nabla \varphi \geq 0\]
a.e.\@ on $H$.  Inserting this into \eqref{Ann:cofactorPartialIntegration} yields for each $R \in (a,k)$ that
\[ a∫_{S_R} \etR \cdot \varphi  \,dS  \geq - \int_{A(a,R)} \det \nabla \varphi \,dx. \]
By a straightforward density argument we can suppose that $\varphi$ is of class $C_{0}^{2}(A)$.    In particular, we can apply Lemma \ref{brahms1}	
to deduce that
\[ a∫_{S_R} \etR \cdot \varphi  \, dS \geq -\frac{1}{2} \int_{S_R} J \varphi \cdot \varphi_{,_{\tau}} \,dS. \]
Changing $R$ to $r$, mutiplying both sides by $\zeta(r)$ , integrating with respect to $r$ over $(a,k)$ and recalling \eqref{Ann:DerivativeIin_etr}, it follows that
\begin{align}\label{paisiello} 2 \langle \nabla u , \nabla \varphi \rangle \geq - \int_{H} \zeta(r) J \varphi \cdot \varphi_{,_{\tau}}\,dx.\end{align} 
The function $\zeta$ is a constant multiple of $1/r^2$, so we focus now on proving that
\[\int_{H} -\frac{1}{r^2}  J \varphi \cdot \varphi_{,_{\tau}}  \, dx = 2 \int_{H} \left(\frac{1}{k} - \frac{1}{r}\right)	\det \nabla \varphi \,dx.\]
This can be seen as follows:
\begin{align*}  \int_{H} -\frac{1}{r^2}  J \varphi \cdot \varphi_{,_\tau}  \, dx & = \int_{0}^{2\pi} \int_{a}^{k} \left(\frac{1}{r}\right)_{,_{r}} J\varphi \cdot \varphi_{,_\theta} \,dr \, d \theta \\
& = \frac{1}{k} \int_{S_{k}} J \varphi \cdot \varphi_{,_{\tau}} \,dS  - \int_{H} \frac{1}{r}J \varphi_{,_{r}} \cdot \varphi_{,_{\tau}} \, dx - \int_{a}^{k}\int_{0}^{2\pi} \frac{1}{r}J\varphi \cdot (\varphi_{,_{r}})_{,_{\theta}} \,d\theta \, dr \\
& = \frac{2}{k} \int_{H} \det \nabla \varphi \,dx - \int_{H} \frac{1}{r}\det \nabla \varphi \, dx  + \int_{H} \frac{1}{r} J \varphi_{,_{\tau}} \cdot \varphi_{,_{r}} \,dx \\
& = 2\int_{H} \left(\frac{1}{k} - \frac{1}{r}\right)	\det \nabla \varphi \,dx.
\end{align*}
Hence 
\[ - \int_{H} \zeta(r) J \varphi \cdot \varphi_{,_{\tau}}\,dx \geq 2\left(1+\frac{\omega^2}{a^2}\right) \int_{H} \left(\frac{1}{k} - \frac{1}{r}\right)	\det \nabla \varphi \,dx,\]
so that, by \eqref{paisiello}, 
\[  2 \langle \nabla u , \nabla \varphi \rangle \geq 2\left(1+\frac{\omega^2}{a^2}\right) \int_{H} \left(\frac{1}{k} - \frac{1}{r}\right)	\det \nabla \varphi \,dx.\]
Inserting this into \eqref{hubertparry} gives
\[ D(v) \geq D(u) + \int_{H} |\nabla \varphi|^2 + 2\left(1+\frac{\omega^2}{a^2}\right) \left(\frac{1}{k} - \frac{1}{r}\right)	\det \nabla \varphi \,dx + \int_{A\setminus H} |\nabla \varphi|^2 \,dx,\]
from which the proof of part (ii) of the Proposition can easily be concluded.  
\end{proof}

This leads naturally to the following result that $u_{\ast}^{N}$ is a minimizer of $I_\infty$ with respect to perturbations with suitably located support.

\begin{corollary}\label{suitsupport}  Let $N\in \mathbb{N}$ and let $u_{*}^{N}$ minimize $I_{\infty}$ in $\tilde{\mathcal{A}}_{N,\,\ssym}$.    Let $v \in \tilde{\sca}_N$ be such that $\varphi:=v-u_{\ast}^N$ has support in the annulus $A(r_{\ast},b) \subset A$, where
\begin{align*}
\frac{1}{r_{\ast}} = \frac{1}{k} + \frac{a^2}{a^2+\omega^2}.
\end{align*}
Then $I_{\infty}(v) \geq I_{\infty}(u_{\ast}^{N})$.
\end{corollary}

\begin{proof}  If $\textrm{spt}\,\varphi$ lies in $A(r_\ast,b)$ as defined then a simple calculation shows that 
\begin{align*}  \left|\left(1+\frac{\omega^2}{a^2}\right)\left(\frac{1}{k} - \frac{1}{r}\right)\right| &\leq 1  \end{align*}
for any $r \leq k$ such that $S_r$ meets $\textrm{spt}\,\varphi$.  Hence, by Hadamard's inequality, which in the $2 \times 2$ case is $2|\det F| \leq |F|^2$, the quantity
\[ |\nabla \varphi|^2 + 2 \left(1+\frac{\omega^2}{a^2}\right)\left(\frac{1}{k} - \frac{1}{r}\right) \det \nabla \varphi \]
is pointwise nonnegative, and hence part (ii) of Proposition \ref{wolfgangAM} implies that $I_{\infty}(u_{\ast}^{N}+\varphi) \geq I(u_{\ast}^{N})$.   
\end{proof}

\subsection{The case $h=h_{0}$: twist minimizers with area compression energy}
\label{Ann:sec:WithVolumeCompression}
We now return to the case also considered by Post and Sivaloganathan \cite{PostSivaloganathan97:HomotopyConditionsExistenceMultipleEquilibria}.  
We seek a minimizer of the functional
\begin{align}
	I_{0}(u) = ∫_A \frac{1}{2}|\nabla u|^2 + h_{0}(\det \nabla u) \, dx
	\label{Ann:minAN}
\end{align}
for each $N\in \mathbb{N}$, but where this time the local invertibility condition $\det \nabla u > 0$ a.e.\@ is encoded in the function $h_0$ via the properties
\begin{itemize}
	\item[(H1)] $h_0$ is convex with $h_0\geq 0$
	\item[(H2)] $h_0 \in C^3( (0,+\infty ) )$ and for some positive constants $s,c_1,c_2$ and $d_0$,
		$c_1d^{-s-k} \leq  (-1)^k h_0^{(k)}(d) \leq  c_2 d^{-s-k}$ for $0<d<d_0$ and $k=0,1,2$
	\item[(H3)] $h_0(d) = +\infty $ for $d\leq 0$
	\item[(H4)] For some real number $\tau $ and positive constants $c_3$, $c_4$ and $d_1$
		$c_3d^\tau  \leq  h_0''(d) \leq  c_4d^\tau $ for $d\geq d_1$.
\end{itemize}
Again, instead of looking at the whole of $\tilde{\mathcal{A}}_N$, we focus on those functions in $\tilde{\mathcal{A}}_N$ that are rotationally symmetric, i.e. we minimize $I_0$ on the set $\tilde{\sca}_{N, \,\ssym}$ 
defined in \eqref{Ann:A0Nsym}.  Using the same notation as in the previous section, and by following  \cite{PostSivaloganathan97:HomotopyConditionsExistenceMultipleEquilibria}, we conclude that the rotationally symmetric minimizer $u_{\ast}^{N}$ of $I_0$ in $\tilde{\sca}_{N, \,\ssym}$ has radial and angular parts $\rho ,\psi$ of class $C^2(a,b)$ and, moreover, that $u_{\ast}^{N}$ solves the Euler-Lagrange equations, which for rotationally symmetric maps simplify to 
\begin{align}
	\nonumber \left[ r\dot{\rho } + \rho h_0'(d) \right]' &= \frac{\rho }{r} + r\rho \dot{ψ}^2 + \dot{\rho }h_0'(d)
	\intertext{and}
	r\rho ^2\dot{ψ} &= \omega.	\label{Ann:ELequation}
\end{align}
In fact, since we assume slightly stronger conditions on $h_0$ than Post and Sivaloganathan do, we actually obtain that  $\rho \in C([a,b])∩C^3(a,b)$.
Since $I_{0}(u_{\ast}^{N})<+\infty$, it is impossible for $\det \nabla u_{\ast}^{N}$ to vanish on a set of positive measure.  However, it may still be possible for $\dot{\rho}(r) = 0$ for some $r$ (where $r=a+$ is understood on the inner boundary and $r=b-$ on the outer), which would correspond to $\det \nabla u_{\ast}^{N}(x)=0$ on the circle $S_r$.  This was the case for each $r \in [a,h]$, for example, in the previous section of the paper.   The following lemma is motivated by the well-known works  \cite{BaumanOwenPhilips91:MaximumPrinciplesAPrioriEstimates,BaumanOwenPhillips91:MaximalSmoothness}.

\begin{lemma}\label{lem:monotonedz}
	Let $N\in \mathbb{N}$, let $u_{\ast}^{N}$ minimize $I_0$ in $\tilde{\sca}_{{N}, \,\ssym}$ and define the function $f: (0,\infty) \to \mathbb{R}$ by $f(s):=sh_0'(s) -h_0 (s)$.  Define the functions $d:=\det \nabla u_{\ast}^{N}$ and $z := \frac{1}{2}|\nabla u_{\ast}^{N}|^2 + f(d)$ on the annulus $A$.  Then $d$ and $z$ depend only on the radial variable $r$, and $d$ is strictly monotonically increasing on $(a,b)$ while $z$ is strictly monotonically decreasing on $(a,b)$.
	\label{th:Ann:dIncreasingZDecreasing}
\end{lemma}
\begin{proof}
A direct calculation using the form of the solution $u_{\ast}^{N}$ shows that $d=\frac{\rho \dot{\rho}}{r}$, which is clearly independent of the angular variable $\theta$.  The same us true of $z$, as is shown in \eqref{Ann:z} below.  From the remarks above (concerning the regularity of $\rho$, essentially applying \cite{PostSivaloganathan97:HomotopyConditionsExistenceMultipleEquilibria}) the quantities $d$ and $z$ are differentiable.  Now assume for a contradiction that $\dot{d} \leq  0$.
	Then 
	\begin{align}
		\ddot{\rho } \leq  \frac{1}{\rho }\left( d - \dot{\rho }^2 \right).
		\label{Ann:ddrhoEstimate}
	\end{align}
The Euler-Lagrange equations \eqref{Ann:ELequation} are equivalent to 
	\begin{align}
		\ddot{\rho }\left( r + \frac{\rho ^2}{r}h_0''(d) \right) = \frac{1}{r}\left( \rho  + \frac{\omega ^2}{\rho ^3} \right) - \dot{\rho } + \frac{\rho }{r}\left( d - \dot{\rho }^2 \right)h_0''(d).
	\end{align}
	The factor $r+\frac{\rho ^2}{r}h_0''(d)$ is always positive, so we can use \eqref{Ann:ddrhoEstimate} on the left-hand side to obtain
	\begin{align*}
		\dot{\rho } + \frac{r}{\rho }\left( d - \dot{\rho }^2 \right) \geq  \frac{1}{r}\left( \rho  + \frac{\omega ^2}{\rho ^3} \right)
	\end{align*}
	Multiplying this through by $\frac{\rho }{r}$ we deduce that
	\begin{align*}
		- \left( \dot{\rho } - \frac{\rho }{r} \right)^2 \geq  \frac{\omega ^2}{r^2\rho ^2},
	\end{align*}
	which is impossible since $\omega ≠0$.
	\par
	For $z$ we have, by direct calculation, 
	\begin{align}
		z &= \frac{1}{2}\left( \dot{\rho }^2 + \rho ^2\dot{ψ}^2 + \frac{\rho ^2}{r^2} \right) + f(d) = \frac{1}{2}\left( \dot{\rho }^2 + \frac{\omega ^2}{r^2\rho ^2} - \frac{\rho ^2}{r^2} \right) + f(d) + \frac{\rho ^2}{r^2}
		\label{Ann:z}
	\end{align}
	Differentiating and using \eqref{Ann:ELequation} we find
	\begin{align*}
		\dot{z} &= -\frac{1}{r}\left( \dot{\rho }^2 + \frac{\omega ^2}{r^2\rho ^2} - \frac{\rho ^2}{r^2} \right) + 2\frac{\rho \dot{\rho }}{r^2} - 2\frac{\rho ^2}{r^3} \\
		&= -\frac{1}{r} \left( \dot{\rho }^2 + \frac{\omega ^2}{r^2\rho ^2} + \frac{\rho ^2}{r^2} - 2\frac{\rho \dot{\rho }}{r} \right)
		= -\frac{1}{r} \left( \left( \dot{\rho } - \frac{\rho }{r} \right)^2 + \frac{\omega ^2}{r^2\rho ^2} \right) < 0.
	\end{align*}
\end{proof}
Now we are in the position to prove the following result, which asserts that 
$\det \nabla u_{\ast}^{N}$ is bounded strictly away from $0$ on $A$.  
\begin{lemma} Let $N\in \mathbb{N}$ and let $u_{\ast}^{N}$ minimize $I_0$ in $\tilde{\sca}_{{N}, \,\ssym}$.  Then if $u_{\ast}^{N}$ is expressed in the form
\begin{align*}
	u_{\ast}^{N}(r,\theta) = \rho (r) e_r( \theta + \psi(r))
	\end{align*}
it holds that $\dot{\rho }\in C([a,b])$ with $\dot{\rho }(a) > 0$ and $\dot{\rho }(b) < \infty $.
	\label{th:Ann:drhoAtBoundary}
\end{lemma}
\begin{proof}
	Since $d$ is monotonic on $(a,b)$, the limits $\lim_{r\to a+}d(r)$ and $\lim_{r \to b-} d(r)$ exist (possibly $+\infty $ for $r→b-$).
	Therefore, the limits $\lim_{r \to a+}\dot{\rho }(r)$ and $\lim_{r \to b-}\dot{\rho }(r)$ also exist, with the same qualification for the case $r \to b-$.  
If $\dot{\rho }(r)$ were to vanish as $r \to a+$ then we would have $d \to 0+$ as $r \to a+$, and hence $f(d)=dh_0'(d) - h_0(d)$ would tend to $-\infty$ as $r \to a+$.  Recalling \eqref{Ann:z}, it follows that $z(r) \to -\infty$ as $r \to a+$.    On the other hand, $z(r)$ is decreasing on $(a,b)$ and certainly finite on that interval, implying in particular that $\lim_{r \to a+} z(r)$ is not $-\infty$, which is a contradiction.   Hence $\dot{\rho}(a+)$ is strictly positive.  The argument needed to show that  $\dot{\rho }(b-) < \infty $ is similar.
\end{proof}

We remark, in passing, that we are able to derive the following maximum principle.

\begin{theorem}
	Let $u_{\ast}^{N}$ satisfy the hypotheses of Lemma \ref{th:Ann:drhoAtBoundary}.  Then the function $\frac{\rho }{r}$ attains no interior local maximum.  In particular, $\left( \frac{\rho }{r} \right)^{\Cdot\Cdot}$ changes sign only once and $\frac{a}{b} \leq  \frac{\rho }{r} < 1$ in $(a,b)$ with $\frac{\rho }{r}=1$ at $a,b$.
	\label{th:MaximumPrinciple_rhoOverR}
\end{theorem}
\begin{proof}
	Assume there exists an $r\in (a,b)$ s.t. $\left( \frac{\rho }{r} \right)^{\Cdot[1.5]} = 0$ and $\left( \frac{\rho }{r} \right)^{\Cdot\Cdot} \leq  0$.
	Then 
	\begin{align}
		0 \geq  \left( \frac{\rho }{r} \right)^{\Cdot\Cdot} = - \frac{2}{r}\left( \frac{\rho }{r} \right)^{\Cdot} + \frac{1}{r}\ddot{\rho } = \frac{1}{r}\ddot{\rho }.
		\label{Ann:ddrhoContradiction1}
	\end{align}
	However, by Theorem~\ref{th:Ann:dIncreasingZDecreasing} we have
	\begin{align}
		0 &< \dot{d} = \frac{1}{r}\left( \dot{\rho }^2 + \rho \ddot{\rho } - d \right) = \frac{\rho }{r}\ddot{\rho } + \dot{\rho }\left( \frac{\rho }{r} \right)^{\Cdot} = \frac{\rho }{r}\ddot{\rho }
		\label{Ann:ddrhoContradiction2}
	\end{align}
	which contradicts \eqref{Ann:ddrhoContradiction1}.
\end{proof}

\section{Shear maps}
\label{She:sect}

In this section we focus on so-called shear maps and their properties.  In brief, for any given domain $D \subset \R^{n}$ a shear map $u_{\sigma} : D \to \R^n$ takes the form 
\[ \us(x) = x + \sigma(x) e,\]
where $e$ is a fixed unit vector in $\R^n$ and the function $\sigma$ is real-valued.  
We echo some of the constructions of Section \ref{sectiontwo} 
by posing and then solving variational problems first in the case that the weak constraint $\det \nabla \us \geq 0$ is required to hold, that is when $h=h_{\infty}$, and then in the case that compression to zero `area' is energetically penalized, corresponding to $h=h_0$.  In the former case, and still in a two dimensional setting, we find conditions which imply that the unique minimizer of a Dirichlet energy among shear maps necessarily satisfies $\det \nabla \us = 0$ on a specified subdomain.  (Cf. Section \ref{twistwithhinfinity} and the `hedgehog map'.)  Moreover, we establish conditions under which the global energy minimizer fails to be $C^1$ at interior points of the domain.   The conditions are based on easily verifiable boundary behaviours of functions harmonic on certain subdomains of $D$.  See Section \ref{ss1} for details.

Where the stronger constraint $\det \nabla \us > 0$ a.e.\@ is required to hold, via $I_0(\us) < +\infty$, we find that even if compression is strongly energetically penalized\footnote{This is achieved by requiring in addition that $\det \nabla \us \geq c > 0$ a.e.\@ in the domain.}, circumstances arise in which the unique energy minimizing shear map fails to be $C^1$.  In this case the gradient is discontinuous `at' certain boundary points.  See Section \ref{ss2} for details.

Our chief ally in proving these assertions is the fact that the Jacobian of any shear map $\us$ is linear in $\nabla \sigma$, viz.
\[ \det \nabla \us = 1+e \cdot \nabla \sigma.\]
Consequently, the Jacobian of a convex combination of any two shear maps $u_{\sigma_1}$ and $u_{\sigma_{2}}$ satisfies
\[ \det \nabla u_{\lambda \sigma_1+(1-\lambda) \sigma_2} = \lambda \det \nabla u_{\sigma_1} + (1-\lambda) \det \nabla u_{\sigma_2}, \]
where $0 \leq \lambda \leq 1$.  In particular, it follows that if the maps $u_{\sigma_i}$ obey the constraint $\det \nabla u_{\sigma_i} \geq 0$ for $i=1,2$ then any convex combination must also obey that constraint.   Moreover,   inserting $F= \nabla \us$ into the general form stored-energy function $W(F)=\frac{1}{2}|F|^2+h_{0}(\det F)$, we find that
\[ W(\nabla \us) = \frac{1}{2}|\1 + e \otimes \nabla \sigma|^2 + h_0(1+e \cdot \nabla \sigma) \]
is in fact \emph{convex} in $\nabla \sigma$.   This convexity turns out to be useful in both the weak and strong constraint cases (corresponding, respectively to the choice $h=h_{\infty}$ and $h=h_0$).  When the weaker constraint $\det \nabla \sigma \geq 0$ a.e.\@ holds, it means that all we need do to establish that a given admissible map is a minimizer is to prove that it is a solution of a variational inequality associated with the energy functional 
\begin{equation}\label{She:iwdef} I_{w}(\sigma) : = \int_{\om}  |\nabla \us|^2 \,dx, \end{equation}
whereas when the strong constraint $\det \nabla \us > 0$ is in force the convexity of $W(\nabla \us)$ in $\nabla \sigma$ allows us to apply elliptic regularity theory under certain conditions, an important intermediate step in determining the behaviour of  $\nabla \us$ near the boundary.  

\subsection{The case $h=h_{\infty}$: shear minimizers without area compression energy}\label{ss1}

For definiteness, we now restrict attention to shear maps applied to the square $Q=[-1,1]^2$ in two dimensions, and we define
\begin{equation}\label{She:genshear} \us(x) = x+ \sigma(x) e_2 \quad \textrm{if} \ x \in Q, \end{equation}
where $e_{2}= (0,1)$.  Define 
\begin{align}\label{She:sigmazero}
\sigma_0(x_1,x_2) & =\left\{\begin{array}{l l} 0 & \textrm{if } \ -1<x_1 \leq 0  \\
-2x_1 x_2 & \textrm{if } \ 0<x_1< \frac{1}{2} \\
-x_2 & \textrm{if } \frac{1}{2} \leq x_1 < 1.\end{array}\right.
\end{align}
Formally speaking, the effect of $u_{\sigma_0}$ is to project the region $P:=\{x \in Q: \ \frac{1}{2} \leq x_1 < 1\}$ onto that part of the $x_1$ axis which it contains.   At the same time, $u_{\sigma_0}$ acts as the identity map on the region $M:=\{x \in Q: \ -1 < x_1 \leq 0\}$.   In the region $N:=\{x \in Q: \ 0 < x_1 < \frac{1}{2}\}$ the map  $u_{\sigma_0}$ brings about a narrowing (in the $x_2$-direction) of $Q$.  Note that, in this notation, $Q=M \cup N \cup P$. Figure \ref{f1} below illustrates both the subdivision of $Q$ and the effect that (a slightly smoothed version of) $u_{\sigma_0}$ has on $Q$.


\begin{figure}[ht]
\psfragscanon
\psfrag{a}{$x_1$}
\psfrag{b}{$x_2$}
\psfrag{M}{$M$}
\psfrag{N}{$N$}
\psfrag{P}{$P$}
\psfrag{c}{$u_{\sigma_0}$} 
\psfrag{g}{$\frac{1}{2}$}
\psfrag{h}{$1$}
\psfrag{i}{$\delta$}
\psfrag{j}{-$\delta$}
\centering
\includegraphics[width=0.8\textwidth]{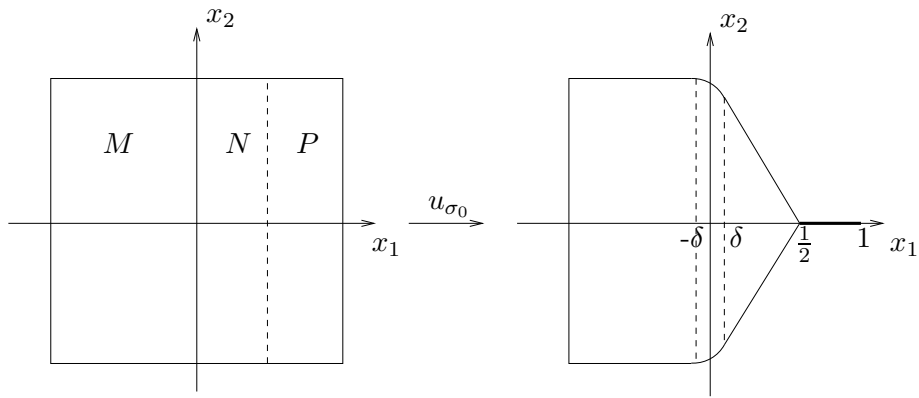}
\caption{The boundary $\partial Q$ is subjected to the displacement $u_{\sigma_0}$.  The regions $N$ and $P$ correspond respectively to `narrowing' and `pinching' respectively.}
\label{f1}
\end{figure} 

We remark that the procedure described below easily adapts to more general boundary conditions than $\sigma_0$:  we use $\sigma_0$ mainly as a convenient means of illustration.   Now define the class of admissible shear maps in the weak constraint case by
\begin{equation}\label{She:defscaw} \sca_{w} = \{ \sigma \in W^{1,2}(Q; \R): \ \us= u_{\sigma_0} \ \textrm{on } \partial Q, \ \det \nabla \us \geq 0 \ \textrm{a.e.} \ \textrm{in }Q \}.\end{equation}
 Here, the boundary conditions are meant in the sense of trace.  

\begin{lemma}\label{lem:eleq} Let $I_w$ and $\sca_w$ be defined by \eqref{She:iwdef}, \eqref{She:defscaw} respectively.  Then $I_w$ has a unique global minimizer in $\sca_w$.  In particular, the global minimizer $\sigma_{w}$ of $I_{w}$ in $\sca_w$ satisfies the inequality
\begin{equation}\label{She:elweak} \int_{Q} \nabla \sigma \cdot \nabla \eta \,dx \geq 0 \end{equation} 
for all $\eta \in W^{1,2}(Q;\R)$ such that $\sigma+\eta$ belongs to $\sca_w$.  
\end{lemma}
\begin{proof} To prove the first assertions of the lemma it suffices to show that $\sca_w$ is nonempty and closed under weak convergence in $W^{1,2}(Q,\R)$ and then to apply the direct method of the calculus of variations.  

A short calculation shows that 
\begin{displaymath}\det \nabla u_{\tilde{\sigma}_0}(x) = \left\{\begin{array}{l l} 1 & \textrm{if } x \in M \\
1-2x_1  & \textrm{if } x \in N \\
0  & \textrm{if } x \in P,\end{array}\right. \end{displaymath} 
where $M$, $N$ and $P$ are as defined above.  Since $0 \leq \det \nabla u_{\tilde{\sigma}_0}=1+\tilde{\sigma}_{0,_2}$ a.e.\@, it follows from standard properties of mollifiers that $1+\sigma_{0,2} \geq 0$ a.e.\@ in $Q$.  Therefore $\det \nabla \sigma_0 \geq 0$ a.e.\@ in $Q$, and so $\sigma_0$ is admissible.  In particular, $\sca_w$ is nonempty.

Now let $\sigma^{(j)}$ be a sequence in $\sca_w$ converging weakly to $\sigma$.  Properties of the trace imply that $\sigma$ satisfies the same boundary conditions as all the $\sigma^{(j)}$, and since $\det \nabla u_{\sigma^{(j)}} = 1+\sigma^{(j)}_{,_{2}} \geq 0$ a.e.\@ in $Q$ for all $j$, it easily follows that $\det \nabla \us \geq 0$ a.e.\@ in $Q$ also.   Thus $\sca_w$ is weakly closed.  The convexity of $I_{w}$ with respect to $\sigma$ coupled with the direct method then yields the existence of $\sigma_w$ minimizing $I_w$ in $\sca_w$.  The minimizer is unique because the functional $I_w$ is strictly convex and the class $\sca_w$ is convex.   

We now prove that \eqref{She:elweak} is necessary and sufficient for $\sigma$ to minimize $I_{w}$ in $\sca_w$.   Let $\eta \in W^{1,2}(Q;\R)$ be such that $\sigma + \eta \in \sca_w$, and let $\sigma$ minimize $I_w$ in $\sca_w$.  Then by writing
\[\sigma + \eps \eta = \eps (\sigma + \eta) + (1-\eps) \sigma \]
and noting that the right-hand side clearly belongs to $\sca_w$ provided 
$0 \leq \eps \leq 1$,  it follows by minimality that $I(\sigma+\eps \eta) \geq I_w(\sigma)$ for all such $\eps$.  Now,
\[ I_w(\sigma) = \int_{Q} |\1 + e_2 \otimes (\nabla \sigma)|^2 \,dx, \]
so that 
\[ \partial_{\eps}\arrowvert_{\eps=0} I_{w}(\sigma+\eps \eta) \geq 0\]
yields the inequality
\begin{equation}\label{Shear:neccond1}\int_{Q} \eta_{,_2} + \nabla \sigma \cdot \nabla \eta \,dx \geq 0 \end{equation} 
for all such $\eta$.  Applying the boundary condition $\eta\arrowvert_{\partial Q}=0$ to this gives \eqref{She:elweak}.   Note that \eqref{She:elweak} is a sufficient condition for the minimality of $\sigma$ in $\sca_w$.     This follows immediately from the identity
\[ I_w(\sigma + \eta) = I_w(\sigma) + \int_{Q} |\nabla \eta|^2 + 2 \eta_{,_2} + 2 \nabla \sigma \cdot \nabla \eta \,dx.\]
\end{proof}

The next result shows that any element $\sigma$ of  $\sca_w$ satisfies $\det \nabla \us(x)=0$ a.e.\@ on $P$, which is in accordance with physical intuition where the region is severely `pinched'. 

\begin{lemma}\label{She:lemvanjac}  Let $\sigma$ belong to $\sca_w$.  Then 
\begin{equation}\label{She:box1}
\sigma_0(x_1,-1) -1- x_2 \leq \sigma(x) \leq \sigma_0(x_1,1) + 1 -x_2 
\end{equation}
for a.e.\@ $x$ in $Q$.   In particular, $\sigma(x) = -x_2$ for a.e.\@ $x$ in $P$, so that $\det \nabla \us = 0$ a.e.\@ on $P$.
\end{lemma}
\begin{proof}  For a.e.\@ $x_1$ in $(-1,1)$ it holds that 
\[ \sigma(x_1,1)-\sigma(x_1,x_2) = \int_{x_2}^{1} \sigma_{,_{2}}(x_1,t)\,dt\]
for a.e.\@ $x_2$ in $(-1,1)$.  Applying the constraint $\sigma_{,_{2}} \geq -1$ and the boundary condition gives 
\[ \sigma(x) \leq \sigma_0(x_1,1)+1-x_2\] 
a.e.\@ $x$ in $Q$.   Arguing similarly, using the boundary condition at points of the form $(x_1,-1)$, we obtain the left-hand inequality in \eqref{She:box1}.   The last assertion of the lemma follows by observing that $\sigma_0(x_1,\pm 1) = \mp 1$ when $x_1 \in (\frac{1}{2},1)$.\end{proof}

There is an interesting and quite subtle interaction between the solution $\sigma(x)=-x_2$ on the region $P$ with its possible behaviour elsewhere on the domain.   This yields a test for whether the constraint $1+\sigma_{,_{2}}\geq 0$ a.e.\@ becomes an equality on a set of positive measure in the subdomain $Q \setminus P$.  In other words, it is possible to test whether $\det \nabla \us = 0$ holds on a set of positive measure \emph{away} from the pinched part $P$ of the domain $Q$, where, by Lemma \ref{She:lemvanjac}, the vanishing of the Jacobian is automatic for all competitors $\sigma$ in $\sca_w$.   

\begin{lemma}\label{lem:slopes} Let $\sigma$ minimize $I_w$ in $\sca_w$ and define 
\begin{align}\label{def:omega}\om:=Q \setminus P.\end{align}   Then at most one of 
\begin{itemize}\item[(i)] $\essinf \{1+\sigma_{,_{2}}(x): x \in U\} > 0$ for all $U \subset \om$ with $\meas{U} > 0$, and
\item[(ii)] $\int_{-1}^{1} \phi(x_2) \sigma_{,_{1}}(1/2,x_2)\,dx_2  = 0$ for all $\phi \in C_c^1((-1,1))$.
\end{itemize}
is true.
\end{lemma}
\begin{proof} Suppose for a contradiction that both (i) and (ii) hold.  Let $B(y,\delta) \subset \om$ and take $\varphi \in C_c^1(B(y,\delta),\R)$.   Then, since by hypothesis there is $c>0$ such that 
$1+\sigma_{,_{2}}(x) \geq c$ for a.e.\@ $x$ in $B(y,\delta)$, it is the case that $\sigma + \eps \varphi$ belongs to $\sca_w$ for all sufficiently small $\eps$.  Arguing as in the prelude to \eqref{She:elweak}, it follows that 
\begin{equation}\label{She:harmonicom}\int_{\om} \nabla \varphi \cdot \nabla \sigma \,dx = 0,\end{equation}
and hence by standard theory, that $\sigma$ is harmonic on the open set $\om$.  

Next, let $\Phi \in C_c^1(Q,\R)$ and note that, since the set $K:=\partial P \cap \partial \om$ has (two-dimensional) Lebesgue measure zero, it follows from \eqref{She:harmonicom}, the final assertion of Lemma \ref{She:lemvanjac} (which implies that $\sigma = -x_2$ on $P$) and Green's theorem that  
\begin{align*}\int_{Q} \nabla \Phi \cdot \nabla \sigma \,dx & =\int_{K} \Phi\left(1/2,x_2\right) \sigma_{,_{1}}\left(1/2,x_2\right)\,dx_2 - \int_{P} \Phi_{,_{2}} \,dx.
\end{align*}
Since $\Phi$ has compact support in $Q$, the second integral on the right-hand side vanishes.   Therefore, since we are assuming that 
(ii) holds, the previous line implies that $\int_{Q} \nabla \Phi \cdot \nabla \sigma \,dx=0$ for all $\Phi \in C^1_c(Q)$, and hence that $\sigma$ is harmonic on $Q$.  But $\sigma=-x_2$ on $P$ by Lemma \ref{She:lemvanjac} and hence, since $\sigma$
 is harmonic on $Q \supset P$ and $P$ has a nontrivial interior, it follows that $\sigma=-x_2$ on all of $Q$.   This requirement violates the boundary conditions, which is a contradiction.
\end{proof}


\begin{proposition}\label{bruch1}  Let $\sigma_0$ be as defined in \eqref{She:sigmazero} and let $\Sigma$ be the unique harmonic function agreeing with $\sigma_0$ on $\partial \om$, where $\om=Q \setminus P$ is defined in \eqref{def:omega}.   
Then
\begin{displaymath}\sigma(x)=\left\{ \begin{array}{l l} \Sigma(x) & \textrm{if } x \in \om \\ 
-x_2  & \textrm{if } x \in P \end{array} \right.
\end{displaymath}
is the unique global minimizer of $I_w$ in $\sca_w$ and $\det \nabla \us > 0$ everywhere in $\om$.  It also holds that $\sigma_{,_1}$ cannot vanish $\mathcal{H}^{1}$-a.e.\@ along the set $K=\{y \in Q: \ y_{1}=1/2\}$.   In particular, the global minimizer does not belong to the class $C^1(Q)$.   
\end{proposition}

\begin{proof}  The first part of the proof consists in showing that $\sigma$ is admissible and that $\det \nabla \us>0$ in $\om$: this is done in Steps $1-3$.  Steps $4$ and $5$ deal respectively with the last two sentences in the statement of the Proposition.   

\vspace{1mm}
\noindent\textbf{Step 1}  By standard results in the theory of harmonic functions, $\sigma$ agrees with $\sigma_0$ in the sense of trace on $\partial Q$, so it only remains to prove that $1+\sigma_{,_{2}} \geq 0$ a.e.\@ in $\om$, this fact being immediate in $P$. Consider $z_1(x) = 1-x_2$ and note that $\Sigma(x) \leq z_1(x)$ for all $x \in \partial \om$.  Since $z_1$ is harmonic and both functions belong to $W^{1,2}(\om)$, the weak maximum principle implies that $\Sigma(x) \leq z_1(x)$ for all $x \in \bar{\om}$.  In particular, $\Sigma(x_1,1+h) \leq -h$ for $-1 < h < 0$ and $-1 \leq x_1 \leq 0$.  Therefore, since $\Sigma(x_1,0)$ for $-1 \leq x_1 < 0$, we have
\begin{align*}
\frac{\Sigma(x_1,1+h) - \Sigma(x_1,1)}{h} & \geq -1
\end{align*}
for this range of $x_1$ and $h$, so that letting $h \to 0$ gives $\Sigma_{,_2}(x_1,1) \geq -1$.  A similar argument using the harmonic function $-1-x_2$, which satisfies $z_2(x) \leq \Sigma(x)$ for $x \in \partial \om$, implies that $\Sigma_{,_{2}}(x_1,-1) \geq -1$ for $-1 \leq x \leq 0$.     The derivatives $\Sigma_{,_{2}}(x_1,1)$ and  $\Sigma_{,_{2}}(x_1,-1)$ for $0 \leq x_1 \leq 1/2$ can be bounded below by $-1$ in a similar fashion, the only differences being that the comparison function $z_1$ should be replaced by $z_1(x) - 2x_1$ in the first case and $z_2(x)$ by $z_2(x)-2x_1$ in the second.  It is immediate from the boundary condition that $1+\Sigma_{,_2}(\pm 1, x_2) \geq 0$, so that, in summary, $1+\Sigma_{,_{2}} \geq 0$ on all of $\partial \om$.
\vspace{1mm}

\noindent\textbf{Step 2}  Now note that $1+\Sigma_{,_{2}}$ is harmonic in $\om$, so that if  $\Sigma_{,_{2}}$ were to belong to $W^{1,2}(\om)$ then the weak maximum principle would apply.   This, together with the previously established fact that $1+ \Sigma_{,_{2}} \geq 0$ on $\partial \om$ would then imply $1+ \Sigma_{,_{2}} \geq 0$ on $\om$, and hence that $\sigma$ belongs to $\sca_w$ as desired.  By \cite[Theorem 8.12]{GT}, $\Sigma$ belongs to $W^{2,2}(\om')$, where $\om'$ is any subset of $\om$ whose closure does not contain the corners $(-1,\pm 1)$, $(1/2, \pm 1)$ or the points $(0,\pm 1)$.   The reason is that away from these points the boundary condition $\Sigma = \sigma_0$ is smooth and the (flat) boundary is sufficiently regular.  In particular, it follows that $\Sigma_{,_2}$ belongs to $W^{1,2}(\om')$ for such $\om'$ and the weak maximum principle will apply. We have already established that $1+\Sigma_{,_{2}} \geq 0$ on $\partial \om$, but it could still be that, for some $c>0$, $1+\Sigma_{,_{2}}<-c$ occurs in $\om$ and persists `up to the corners':  the argument we give below rules this out.

To fix ideas, let $\eps > 0$, let $C=\{(-1,\pm 1), (1/2,\pm 1), (0,\pm 1)\}$ and define
$\om_{\eps} = \om \setminus \cup_{a \in C} B(a,\eps)$.  Thus $\om_{\eps}$ is a version of $\om$ with small neighbourhoods of the set $C$ removed.   Each point $a$ in $C$ has now given rise to two distinct corners $a_1$ and $a_2$, say,  on $\partial \om$, but it is easy to smoothen $\partial \om_{\eps}$ near the newly created corners, thereby producing a new subset $\om'_{\eps}$, say, of $\om_{\eps}$ with the properties that
(i) $\partial \om'_{\eps}$ is smooth and (ii) $\partial \om'_{\eps}$ agrees with $\partial \om$ except possibly in sets of the form $B(a,2\eps)$, where $a$ lies in $C$.  Thus 
\begin{align*}
\partial \om \setminus \partial \om'_{\eps} = \bigcup_{a\in C}\Gamma^{\eps}_{a}
\end{align*}
where each $\Gamma^{\eps}_a$ is a smooth curve whose maximum distance from $\partial \om$ is of order $2 \eps$.   
\vspace{2mm}

\noindent \textbf{Claim:} for each $a$ in $C$ it is the case that 
\begin{align*}\liminf_{\eps \to 0}\inf_{\Gamma^{\eps}_{a}} (1+\Sigma_{,_{2}}) \geq 0.
\end{align*}

\vspace{1mm}
\noindent \textbf{Proof of claim:} in the notation of Lemma \ref{downampney}, let $E=\{x \in \om: \ \Sigma_{,_2}(x) + 1 < 0\}$.   Without loss of generality let $a=(-1,1)$ and let $z \in \Gamma_{a}^{\eps}$.   Suppose that $\Sigma_{,_2}(z) + 1 < 0$ and for any $y$ in $\om$ let $Py=(y_1,1)$, that is, $Py$ is the projection of $y$ onto the upper boundary of $\om$.    Since $\Sigma$ is smooth on compact subsets of $\om$, for any $y$ in $\om$ there exists  a first point $y_1(y)$ on the line $[y,Py]$ where $F:=\Sigma_{,_2}+1$ satisfies $F(y_1(y)) \geq 0$.  In particular, $[y, y_{1}(y)) \subset E$.   Note that $|y-y_1(y)|\leq d(y):= \dist(y,\partial \om)$.    Then we estimate $F(y)$ from below as follows:   
\begin{align*}
F(y)& = \int_{0}^{1} \nabla F((1-t)y+ty_1(y)) \cdot (y_1(y)-y) \,dt \\
& \geq - d(y) \int_{0}^{1} |\nabla F((1-t)y+ty_1(y))|\,dt.\end{align*}
Now let $0<r < \frac{1}{2}d(z)$ and integrate over $B(z,r)\subset \om$.  Note that the bounds $d(z)/2 < d(y) < 3d(z)/2$ are immediate for $y \in B(z,r)$. Since $F$ is harmonic, the mean value theorem applies, so that
\begin{align*} F(z) & \geq -\frac{3d(z)}{2 \pi r^2} \int_{B(z,r)}\int_{0}^{1}|\nabla F((1-t)y+ty_1(y))|\,dt \,dy \\
& \geq -\frac{C d(z)}{r} \left(\int_{T} |\nabla F|^2\,dy\right)^{\frac{1}{2}}
\end{align*} using H\"{o}lder's inequality, where $C$ is a positive constant that does not depend on the quantities elsewhere in the estimate.    Here, the set $T\subset E$ is formed from the union of lines $[y,y_{1}(y)]$ where $y \in B(z,r)$ and, by inspection, its measure is at most of order $r d(z) $.  Now suppose we fix $r=d(z)/2$:  then the measure of $T$ is bounded above by a quantity of order $d(z)^2 \leq 4\eps^2$.  Hence the estimate above gives
\begin{align*} F(z) & \geq - 2 C\left(\int_{T} |\nabla F|^2\,dy\right)^{\frac{1}{2}}.
\end{align*} 
Since $T \subset E$, Lemma \ref{downampney} applies and ensures that the integral on the right tends to 0 as $\eps \to 0$.    This proves the claim.

To conclude Step 2 we apply the weak maximum principle to the domain $\om'_{\eps}$ defined above, giving for each $a \in C$
\begin{align*}
\Sigma_{,_{2}}(x)+1 \geq \min\left\{0,\,\inf_{\Gamma_{a}^{\eps}} 1+\Sigma_{,_{2}}\right\} \ \forall x \in \om'_{\eps}.  
\end{align*}
Letting $\eps \to 0$ and applying the claim above, we see that $\Sigma_{,_{2}}(x)+1 \geq 0$ for any $x \in \om$.

\vspace{1mm}
\noindent \textbf{Step 3} We apply the strong maximum principle to establish that $\Sigma_{,_{2}}(x)+1>0$ in $\om$.   Suppose for a contradiction that there is $x^*$ in $\om$ such that $\Sigma_{,_{2}}(x^*)+1=0$.  Pick a subdomain $\widehat{\om}=(-1,1/2) \times (1-s,-1+s)$ containing $x^*$ and where $s>0$.  Notice that $\Sigma_{,_2}(1/2,x_2)+1=0$ for $|x_2|< 1$, so that $x^*$ would be an interior minimum for $\Sigma_{,_{2}}+1$ and $\Sigma$ lies in $W^{2,2}(\widehat{\om})$.  By the strong maximum principle, this is only possible if $\Sigma_{,_{2}}+1=0$ throughout $\om$.  But this violates the boundary condition $\Sigma(-1,x_2)=0$ for $|x_2|<1$.

\vspace{1mm}
\noindent \textbf{Step 4}
Next, we show that $\sigma$ as defined satisfies inequality \eqref{She:elweak}, which, by Lemma \ref{lem:eleq}, is both necessary and sufficient for $\om$ to minimize $I_{w}$ in $\sca_w$.   Let $\eta \in W^{1,2}_{0}(Q)$ be such that $\sigma + \eta \in \sca_w$.  Let $\eta^{(j)}$ approximate $\eta$ in $W^{1,2}$ norm, where $\eta^{(j)} \in C_c^{\infty}(Q)$ for all $j$.   By construction, $\sigma$ is harmonic on each of the subsets $\om$ and $P$, so that, arguing as in the proof of Lemma \ref{lem:slopes},
\begin{align*} \int_{Q} \nabla \eta^{(j)} \cdot \nabla \sigma \,dx & = \int_{K} \eta^{(j)} \sigma_{,_{1}} \,d\sch^{1} - \int_{P} \eta_{,_{2}}^{(j)} \,dx \end{align*}
The second integral on the right-hand side vanishes trivially.   To deal with the integral along $K=\partial \om \cap \partial \om$ we note that, by Lemma \ref{She:lemvanjac}, we must have $\eta\arrowvert_{L_{x_1}} = 0$ on almost every part line $L_{x_1}=\{x_1\} \times [-1,1]$ in $P$.  Therefore, without loss of generality, we may assume that $\eta\arrowvert_{K}=0$.   Moreover, since $\eta^{(j)} \to \eta$ in particular in $W^{1,2}(\om)$, properties of the trace imply that $\eta^{(j)} \to 0$ in $L^{2}(K)$.   By construction, $\sigma_{,_{1}}$ is bounded on $Q$,  so it follows that $\int_{K} \eta^{(j)} \sigma_{,_{1}} \,d\sch^{1} \to 0$ as $j \to \infty$.  Hence  inequality \eqref{She:elweak} holds as an equality, and it follows from Lemma \ref{lem:eleq} that $\sigma$ as constructed is the global minimizer of $I_{w}$ in $\sca_{w}$.

\vspace{1mm}
\noindent \textbf{Step 5}	
The final assertion of the proposition follows by applying Lemma \ref{lem:slopes}.   Indeed, alternative (i) of that lemma holds because, as we have seen, $\det \nabla \us$ is strictly positive and continuous on $\om$.  Therefore alternative (ii) \emph{cannot} hold, meaning that $\sigma_{,_{1}}$ is not zero when viewed as the trace of $\sigma_{,_{1}} \arrowvert_{\om}$ along $K$.  Since $\sigma_{,_{1}}$ clearly vanishes in $P$, it cannot be that $\nabla \sigma$ is continuous across $K$.     This concludes the proof. 
\end{proof}

\begin{remark}\emph{The last line of the statement of the proposition could be anticipated by noting that $\sigma$ maps the set $K$ to a point.  Therefore in any left-neighbourhood of $K$, with obvious notation, the derivative $\sigma_{,_1}$ could not possibly agree with the same derivative in the region $P$.}
\end{remark}

\subsection{The case $h=h_{0}$: shear minimizers with area compression energy}\label{ss2}

In this section we examine the effect of imposing the constraint $\det \nabla \us >0$ a.e.\@ in $Q$.   We focus in particular on a problem where a displacement boundary condition is applied across a strict subset 
\begin{align}\label{partialQ1}
\partial Q_1 & = \{ x \in \partial Q: \ x=(\pm 1, x_2), \ |x_2| \leq 1\}
\end{align}
of $\partial Q$.  On the `free boundary' $\partial Q \setminus \partial Q_1$ a natural so-called traction-free condition should arise, but this is not straightforward since it involves the first derivatives of $\sigma$ and these are not necessarily defined even in the sense of trace on $\partial Q$.  We make sense of this by imposing on the minimizing $\sigma$ the additional condition $1+\partial_2 \sigma \geq c > 0$ a.e.\@ in $Q$ for some constant $c$, that is we strengthen $\det \nabla \us > 0$ a.e.\@ in $Q$ to $\det \nabla \us \geq c > 0$ a.e.\@ in $Q$.  The convexity of $W(\nabla \us)$ in $\nabla \sigma$ then allows us to apply a bootstrapping argument to improve the regularity of $\sigma$ to $W^{2,2}(Q)$, so that the natural boundary condition is well-defined via the trace theorems for Sobolev functions.   

One outcome of this is that $\sigma$ satisfying these assumptions cannot be $C^{1}(\bar{Q})$:  the `corner' of the domain together with the natural and imposed boundary conditions combine to form a discontinuity in the gradient `at' the corner.  On closer inspection the same phenomenon could be induced by considering a suitable Neumann problem for the Dirichlet energy on the same domain and with the same boundary conditions.  More interesting is its interpretation in the original nonlinear elasticity setting, namely that if a minimizer is such that $\det \nabla \us$ is bounded away from zero a.e.\@ then it is not $C^{1}(\bar{Q})$.  This seems strange because one normally thinks of the condition $\det \nabla \us \geq c > 0$ a.e.\@ as being `regularizing', and indeed we shall see that it is so at interior points of the domain.  We have to conclude that the free boundary $\partial Q_2:=\partial Q \setminus \partial Q_1$ plays a significant role in producing the discontinuity in $\nabla \sigma$ `at' the boundary.  

We now give the details of the results alluded to above.  Let 
\begin{equation}\label{She:defis} I_{s}(\sigma) = \int_{Q} \frac{1}{2}|\nabla \us|^2+h_0(\det \nabla \us)\, dx,\end{equation}
where, for concreteness, we assume that $h_0$ satisfies hypotheses (H1)-(H3).   Let $\partial Q^{\pm}=\{(\pm 1,t): \ -1 \leq t \leq 1\}$ denote the left (-) and right (+) sides of $Q$, and let $\sigma_1$ be any $W^{1,2}(Q;\R)$ map such that $I_{s}(\sigma_1) < +\infty$.  Finally, define the class of admissible maps in the strong constraint case by
\begin{equation} \label{She:defscas}\sca_{s} = \{ \sigma \in W^{1,2}(Q; \R): \ \us= u_{\sigma_1} \ \textrm{on } \partial Q_{+} \cup \partial Q_{-}, \  I_s(\sigma) <+\infty\}.\end{equation}  
Note that, in the notation introduced above, $\partial Q_1 = \partial Q^{+} \cup \partial  Q^{-}$.

\begin{proposition} Let $I_s$ and $\sca_s$ be defined as in \eqref{She:defis} and \eqref{She:defscas} respectively.  Then there exists a minimizer of $I_s$ in $\sca_s$.
\end{proposition}
\begin{proof} By hypothesis, $\sca_s$ contains $\sigma_1$ and is thus nonempty.  The integrand of the functional $I_s$ is polyconvex and, moreover, satisfies the hypotheses of \cite[Theorem 6.1]{BM84} in the two dimensional case.  Hence $I(u):=\int_{Q}W(\nabla u)\,dx$ is sequentially lower semicontinuous with respect to weak convergence in $W^{1,2}(Q;\R^2)$.   Let $u_{\sigma^{(j)}}$ be a minimizing sequence which without loss of generality we can suppose to be weakly convergent to $u$, say.  It is straightforward to show that $u=\us$ for some $\sigma$, i.e. $u$ is a shear map, and, by the sequential weak lower semicontinuity of $I(\cdot)$, that $\sigma$ minimizes $I_s$ in $\sca_s$. \end{proof}

For concreteness we fix $\sigma_1=0$, so that $\sigma=0$ (in the sense of trace) on $\partial Q_1$ for any $\sigma$ in $\sca_s$.  This corresponds to applying the boundary condition $\us = \id$ on $\partial Q_1$.  It is clear that $\sigma_1$ is such that $I(\sigma_1)<+\infty$.    

We also impose a further condition on the convex function $h_0$: namely, that the upper bound in condition (H4) defined in Section \ref{Ann:sec:WithVolumeCompression} holds with the parameter $\tau_0=0$.  Alternatively, we can (and do) impose the following condition:
\begin{align}\label{ice1port1}  \quad \quad \quad \quad \forall \mu > 0 \quad \exists C_\mu > 0 \quad \forall s \in [\mu,+\infty) \quad \quad \quad |h_0''(s)| \leq C_\mu.
\end{align}
This, together with the next lemma,  will allow us to apply some elliptic regularity theory techniques.
\begin{lemma}\label{lem:gc} Let the $C^2$ function $h_0$ satisfy hypothesis \eqref{ice1port1}.  Then for each $\mu > 0$ there is $C'_{\mu} >0$ such that $|h_0'(s)| \leq C'_{\mu} s$ for all $s \geq \mu$.  
\end{lemma}
\begin{proof}  Using hypothesis \eqref{ice1port1} and the assumption that $h_0$ is $C^2$, it is straightforward to check that
\[ |h_0'(s)| \leq |h_0'(\mu)| + C_\mu |s-s_0|\]
provided $s \geq \mu$.     Therefore
\[ |h_0'(s)| \leq \left(2C_{\mu} + \frac{|h_0'(\mu)|}{\mu}\right) s\]
for all $s \geq \mu$, and the lemma follows.
\end{proof}

We are now in a position to improve the regularity of the minimizing map $\sigma$. In the rest of this section it will be convenient to switch notation, writing $\partial_1 \sigma$ in place of $\sigma_{,_1}$, and so on. 

\begin{lemma}\label{lem:pen} Let $W$ be given by \eqref{www} with $h=h_0$ and assume that $h_0$ is strongly convex, $C^2$ where it is finite, and that it satisfies (H1) - (H3) and \eqref{ice1port1}.   Then the function 
\[ \nabla \sigma \mapsto W(\nabla \us) \]
is strongly convex and the minimizer $\sigma$ of $I_s$ in $\sca_s$ is unique.
Moreover, if there exists $c>0$ such that 
\begin{equation}\label{sp1} 1+\partial_2 \sigma(x) \geq c \textrm{ a.e.} \ x \in Q
\end{equation}
then $\sigma$ belongs to $W^{2,2}(Q\setminus V)$ where $V$ is any compact set whose interior contains the corners $\{(\pm 1, \pm 1)\}$ of $Q$.
\end{lemma}

\begin{proof}  The first assertion of the lemma is straightforward when we see that the convexity of
\[ W(\nabla \us) = \frac{1}{2}|\1 + e_2 \otimes \nabla \sigma|^2 + h_0(1+\partial_2 \sigma)\]
with respect to $\nabla \us$ is equivalent to the strong ellipticity of the system \eqref{syst1} introduced below, so in anticipation of that result we do not prove strong convexity here.

   If there were two distinct minimizers of $I_s$ in $\sca_s$, $\sigma$ and $\bar{\sigma}$, say, with $I_s(\sigma) = I_s(\bar{\sigma})=m$, then the strict convexity of $W(\nabla \us)$ in $\nabla \sigma$ coupled with the convexity of the class $\sca_s$ clearly implies that 
\[ I_s (\sigma / 2 + \bar{\sigma} /2) < m,\]
a contradiction.  Thus $\sigma$ is unique.

Now suppose that condition \eqref{sp1} holds.  Then if $\eta$ is any smooth function with compact support in $Q$ it follows that $\sigma+\eps \eta$ is admissible provided $\eps$ is sufficiently small.  Hence, on using a suitable dominated convergence theorem, it can be checked that $\partial_\eps I_{s}(\sigma+\eps \eta)$ vanishes at $\eps=0$, leading to 
\begin{equation}\label{syst1} \int_{Q} L(\nabla \sigma) \cdot \nabla \eta\,dx = 0, \end{equation}
where 
\[ L(p) = (p_1,1+p_2+ h'(1+p_2)) \quad \forall  p \in \R^2.\]
The hypotheses on $h$ together with assumption \eqref{sp1} imply that \eqref{syst1} is an elliptic system satisfying controllable growth conditions.  To see this, note that by the convexity of $h_0$ and \eqref{sp1}, 
\[ \xi^T DL(p) \xi  =  \xi_1^2 + (1+h_0''(1+p_2)) \xi_2^2 \geq \lambda |\xi|^2\]  
for some $\lambda > 0$ and all $\xi \in \R^2$.  Moreover, by \eqref{sp1} and Lemma \ref{lem:gc}, $|L(p)| \leq C|p|$ for all $p$ such that $1+p_2 \geq c$.  A differencing argument, such as the one given in the course of the proof of \cite[Theorem 1.1, Chapter II]{Giaquinta}, can now be employed to prove that $D^2\sigma \in L^{2}_{\textrm{loc}}$.  In fact, the argument leading to \cite[Proposition 3.1, Chapter VI]{Giaquinta} shows that $\sigma$ belongs to $W^{2,q}(Q,\R)$ for some $q > 2$ (this makes use of reverse H\"{o}lder inequalities derived from the elliptic system \eqref{syst1}).  In particular, $\nabla \sigma$ is H\"{o}lder continuous on any compact subset of $Q$. Moreover, \eqref{syst1} can now be written as
\begin{equation}\label{syst2} \partial^2_1 \sigma + (1+h_0''(1+\partial_2 \sigma)) \partial^2_2 \sigma = 0 \quad \textrm{a.e.\@ in} \ Q.\end{equation}
It will be useful below to note that the strong convexity of $h$ together with Lemma \ref{lem:gc} and assumption \eqref{sp1} imply that there are positive constants $c_1$ and $c_2$  such that $c_1 \leq h_0''(1+\partial_2 \sigma(x)) \leq c_2$ holds on $Q$.

The regularity asserted in the lemma is $W^{2,2}(Q\setminus V)$, where $V$ is described above, so we must consider the behaviour near boundary points.   Let $x_0 \in \partial Q$ be such that $x_0 \notin V$.  If $x_0 \in \partial Q_1$ where the boundary condition $\sigma = 0$ is applied, then one can proceed as in the proof of \cite[Theorem 8.12]{GT}.   Specifically, differencing shows that both $\partial_2 \partial_1\sigma$ and $\partial^2_2 \sigma$ belong to $L^{2}(B(x_0,r) \cap Q)$ for all sufficiently small $r$.   Equation \eqref{syst2} then implies that $\partial^2_1 \sigma$ also belongs to $L^{2}(B(x_0,r) \cap Q)$.  The argument needed when $x_0$ belongs to $\partial Q_2=\partial Q \setminus \partial Q_1$ is similar.    A covering argument now implies that $D^2\sigma$ belongs to $L^2(Q \setminus V)$, as required.
\end{proof}  
	
\begin{proposition}  Let $W$ be given by \eqref{www} satisfy all the assumptions of Lemma \ref{lem:pen}, and in addition suppose that $1+ h_0'(1) \neq 0$.  Let $\sigma$ be the unique minimizer of $I_s$ in $\sca_s$ and suppose there is a constant $c>0$ such that $1+\partial_{2}\sigma(x) \geq 0$ for a.e.\@ $x$ in $Q$.  Then $\nabla \sigma$ is not continuous at the corners of $Q$.   
\end{proposition}

Before giving the proof we remark that the condition $1+ h_0'(1) \neq 0$ is tailored to the choice of Dirichlet boundary condition $\sigma=0$ on $\partial Q_1$.  In general, one could easily adapt the condition on $h_0'$, which is not especially restrictive, to reflect a different choice of boundary condition.

\begin{proof} By Lemma \ref{lem:pen} and properties of the trace for Sobolev functions, the trace of $\nabla \sigma$ belongs to $L^{2}(A)$ where $A$ is any measurable subset of $\partial Q$ whose closure does not contain the corners of $Q$.   Green's theorem can now be applied to \eqref{syst1}, yielding
\[ \int_{\partial Q_2} L(\nabla \sigma)\cdot \nu \, \eta\, d\sch^1 = 0  \]    
for any $\eta$ whose compact support does not meet $\partial Q_1$, where $\nu$ is $\pm e_2$ are the only two possible outward pointing normals.  In particular, it follows that
\[ L_2(\nabla \sigma) = 0 \ \textrm{a.e.\@ on } \partial Q_2,\] 
that is 
\begin{equation}\label{perk1} 1+ \partial_2 \sigma(x_1,\pm 1)+ h_{0}'(\partial_2 \sigma(x_1,\pm 1) + 1) = 0 \ \textrm{ a.e. } x_1 \in (-1,1).
\end{equation} 
On the other hand, the boundary condition on $\partial Q_1$ implies  
\begin{equation}\label{perk2} 1+\partial_2 \sigma(\pm 1, x_2)+ h_{0}'(\partial_2 \sigma(\pm 1, x_2) + 1) = 1+h_0'(1) \ \textrm{ a.e. } x_2 \in (-1,1).
\end{equation} 
Therefore if $\nabla \sigma$ were continuous at the corner $(1,1)$, say, then 
\begin{align*} \lim_{x_1 \to 1} 1+\partial_2 \sigma(x_1, 1)+h_0'(\partial_2 \sigma(x_1, 1) + 1) & =  \lim_{x_2 \to 1}1+\partial_{2} \sigma( 1, x_2)+h_0'(\partial_{2} \sigma( 1, x_2) + 1) 
\end{align*}
would necessarily hold, which is impossible because the left-hand side is $0$ by \eqref{perk1} and the right-hand side is $1+h_0'(1) \neq 0$ by \eqref{perk2} and the hypothesis on $h_0'(1)$.   
\end{proof}

\setcounter{section}{0}
\setcounter{equation}{0}
\renewcommand{\theequation}{\thesection.\arabic{equation}}

\appendix
\section*{Appendix}
\renewcommand{\thesection}{A} 

The following result is needed in the proof of Proposition \ref{bruch1}.

\begin{lemma}\label{downampney} Let $\om=(-1,1/2) \times (-1,1)$ and define $\sigma_0$ as per \eqref{She:sigmazero}.  Let $\Sigma$ be the unique harmonic function agreeing with $\sigma_0$ on $\partial \om$, and let $E=\{y \in \om: \ 1+\Sigma_{,_2}(y) < 0\}$. 
Then $\nabla \Sigma_{,_2}$ belongs to $L^{2}(B(a,\gamma) \cap E)$, where $0 < 2\gamma < 1/4$ and $a$ is any point in the set $C:=\{(-1,\pm 1), (1/2,\pm 1), (0,\pm 1)\}$. \end{lemma}
\begin{proof}  We deal first with the case that $a$ is a corner of $\om$, and without loss of generality take $a=(-1,1)$.   Since $\Sigma(x)=0$ for $x \in B(a,2\gamma) \cap \partial \om$, we may extend $\Sigma$ by zero outside $\om$.  Let $\eta \in C_c^1(B(a,2\gamma))$ and define the test function 
\begin{align}\label{test}
\varphi(x) & = \eta^2(x)\min\{\Delta^{2,h}\Sigma(x)+1, 0\} 
\end{align} 
for $x \in \om$ and $h \in (-h_0,h_0)$, where $h_0$ is suitably small.  Here, 
\begin{align*}\Delta^{2,h}\Sigma(x_1,x_2)= \frac{\Sigma(x_1,x_2+h)-\Sigma(x_1,x_2)}{h}
\end{align*}
is the difference quotient in the $e_2 = (0,1)$ direction.  According to the proof of Proposition \ref{bruch1},  $\Sigma(x_1, 1+h) - \Sigma(x_1,1) \leq -h$ for $h < 0$, so that in particular $\Delta^{2,h}\Sigma(x_1,1)+ 1 \geq 0$.  For positive $h$ the difference quotient is zero, so it follows that $\Delta^{2,h}\Sigma(x_1,1)+ 1 \geq 0$ for $-h_0 < h < h_0$ and hence that $\varphi(x)=0$ for $x \in B(0,2\gamma) \cap \partial \om$.  Thus $\varphi \in W_0^{1,2}(B(0,2\gamma) \cap \om)$ and since $\Sigma$ is harmonic in this set we must have $\langle \nabla \Sigma, \nabla \varphi \rangle = 0$, where $\langle \cdot, \cdot \rangle$ is the $L^2(\om)$ inner product.   The standard procedure is now to `difference' this inner product, which leads to 
\begin{align*}\int_{\om} \Delta^{2,h}(\nabla \Sigma) \cdot \nabla \varphi \,dx = 0.   
\end{align*}
Inserting $\varphi$ and applying standard inequalities (see, for example, the proof of \cite[Theorem 8.12]{GT}), we obtain
\begin{align} \label{est1} \int_{\om} \eta^2 |\Delta^{2,h}(\nabla \Sigma)|^2 \chi^2_{E^h}\,dx \leq C
\int_{\om} |\nabla \eta|^2 |\Delta^{2,h}\Sigma + 1|^2\chi^2_{E^h} \,dx 
\end{align}
for some constant $C$ that is independent of $\om$, $h$ and $\Sigma$, and where $E^h = \{y \in \om: \ \Delta^{2,h}\Sigma(y) + 1 < 0\}$ and $\chi_{E^h}$ its characteristic function.    

Now let $y \in E$.  Since $\Sigma$ is harmonic it is smooth in $\om$, so it follows that there is $\rho_y > 0$ and $h_y > 0$ such that $B(y,\rho_y) \subset E^h$ for all $h \in (-h_y,h_y)$.   In particular, $\chi_{E^h} \to \chi_{E}$ pointwise almost everywhere in $B(0,2\gamma) \cap \om$.   Take $\eta$ to be a cut-off function satisfying $\eta(z)=1$ if $z \in B(a,\gamma)$ and $|\nabla \eta| \leq c/\gamma$ for some fixed constant $c$.   Using this and Nirenberg's Lemma (see \cite[Lemma 7.24]{GT}, for example), we obtain from \eqref{est1} that 
\begin{align*}
\int_{E} \eta^2 |\nabla \Sigma_{,_{2}}|^2\,dx \leq \frac{C'}{\gamma^2} \int_{B(a,2\gamma) \setminus B(a,\gamma)} 1+|\nabla \Sigma|^2 \,dx.\end{align*}
This proves the lemma in the case that $a$ is a corner of $\om$.

Now suppose $a=(0,1)$,  let $\eta$ be as above and extend $\Sigma$ by zero outside $\om$ in the region $\{x \in \R^2: x_1 < 0\}$ and by $-2x_1$ in the region $\{x \in \om: x_1 > 0\}$, that is we let $\Sigma(x_1,x_2)=\sigma_0(x_1,1)$ if $x_2>1$.   Using the same test function as defined in \eqref{test} together with the fact established in Proposition \ref{bruch1} that $\Sigma(x_1,1+h)-\sigma_0(x_1,1) \leq -h$ for $h < 0$,  it again follows that $\nabla \Sigma_{,_{2}} \in L^{2} (B(a,\gamma) \cap E)$.  This concludes the proof of the lemma.
\end{proof}

\newpage
	\bibliographystyle{abbrv}
\bibliography{./references}
	\end{document}